\documentclass[10pt,reqno]{amsart}
 
\usepackage[T1]{fontenc}
\usepackage{slashed}
\usepackage{amssymb}
\usepackage{amsmath,amsthm}
\usepackage{amsfonts}
\usepackage{leftidx}
\usepackage{mathrsfs}
\usepackage{enumerate}	
\usepackage{abstract}
\usepackage{stmaryrd}
\usepackage{ulem}
\usepackage{graphicx}
 \usepackage{hyperref}

\usepackage[top=1in, bottom=1in, left=0.7in, right=0.7in]{geometry}

\def\R{\mathbb{R}}

\def\Z{\mathbb{Z}}

\def\p{\partial}

\def\vo{\vspace{1\baselineskip}}

\def\be{\begin{equation}}
\def\ee{\end{equation}}

\newtheorem{theorem}{Theorem}[section]

\newtheorem{lemma}{Lemma}[section]
\newtheorem{proposition}{Proposition}[section]
\theoremstyle{definition}

\theoremstyle{remark}
\newtheorem{remark}{Remark}[section]

\numberwithin{equation}{section}
\begin{document}
 \title[3D Relativistic Vlasov-Poisson]{Global solution of the   3D relativistic Vlasov-Poisson system for a class of large data}
\author{Xuecheng Wang}
\address{YMSC, Tsinghua  University, Beijing, China,  100084
}
\email{xuecheng@tsinghua.edu.cn,\quad xuechengthu@outlook.com 
}

\thanks{}
 
\maketitle
\begin{abstract}
For a class of arbitrary large  initial data with   radial symmetry or cylindrical symmetry,   we prove the  existence of global solutions for the $3D$ relativistic Vlasov-Poisson system for  the plasma physics case. The compact support assumption is not imposed for   both cases. The essential lower bound assumption of the angular momentum in the previous work of Glassey-Schaeffer  \cite{glassey2}   is  not imposed on the initial data for the  cylindrical symmetry case. 
\end{abstract}

\section{Introduction}

We are interested in  the large data problem for  the $3D$ relativistic Vlasov-Poisson (RVP) system in the  plasma physics case, which reads as follows. 
\be\label{vlasovpo}
(\textup{RVP})\quad \left\{\begin{array}{c} 
\p_t f + \hat{v} \cdot \nabla_x f + E \cdot \nabla_v f =0,\quad E=\nabla\phi \\
\Delta \phi = \rho(t), \quad \rho(t):= \displaystyle{\int_{\R^3} f(t,x,v) d v},\quad f(0,x,v)=f_0(x,v),\\ 
\end{array}\right. 
\ee
where $\hat{v}:=v/\sqrt{1+|v|^2}$, $f(t,x,v)\geq 0$ denotes the distribution of particles and $\rho(t,x)$ denotes the density of particles.

  From (\ref{vlasovpo}), the backward characteristics associated with the Vlasov-Poisson  equation read  as follow, 
\be\label{characteristicseqn}
\left\{\begin{array}{l}
\displaystyle{\frac{d}{d s} X(s;t,x,v) = \widehat{V}(s;t,x,v)}, \quad 
\displaystyle{\frac{d}{d s} V(s;t,x,v) = \nabla_x\phi(s,X(s;t,x,v)) }  \\
X(t;t,x,v)=x, \quad V(t;t,x,v)=v.\\ 
\end{array}\right. 
\ee
For   convenience in notation, if without  causing confusion, we usually drop the dependence of characteristics with respect to $(t,x,v)$, which is fixed for most of time, and abbreviate characteristics as $(X(s), V(s)).$
 
  For the RVP system (\ref{vlasovpo}), the following conservation laws hold, 
\be\label{conservationlaw}
E(t):=\int_{\R^n} |\nabla_x \phi(t)|^2 +   \int_{\R^3}  \int_{\R^3} |v|  f(t,x,v) d x d v =E(0),  \quad 
 \| f(t,x,v)\|_{L^p_{x,v}}=   \| f(0,x,v)\|_{L^p_{x,v}} ,
\ee
where $p\in[1,\infty].$

There is a large literature devoted to the study of the non-relativistic case, i.e., with $\hat{v}$ replaced by $v$ in (\ref{vlasovpo}). 
We do not try to elaborate it here but refer readers to Anderson \cite{anderson} and Mouhot \cite{mouhot1} and references therein for more detailed introduction.   A remarkable result by Lions-Perthame \cite{lions} (see also \cite{Pfaffelmoser}) showed that the \textbf{ non-relativistic} Vlasov-Poisson system    in  the plasma physics case  admits global classic solution for very general initial data.

However, the literature on the study of the relativistic case is relatively small. Whether the analogue of Lions-Perthame \cite{lions} holds for $3D$ RVP remains an open problem. Indeed, the conservation law  on the momentum, see (\ref{conservationlaw}), in the relativistic case is much weaker than the non-relativistic case.

If the initial data is smooth and small, then the system (\ref{vlasovpo}) admits global solution. Moreover, the regularity of initial data can be propagated, and the density and its derivatives decay sharply over time, see \cite{ionescu,wang}.

However, the picture of   the large data problem of RVP  seems far from complete. A well-known result by Glassey-Schaeffer \cite{glassey1} says that the RVP system (\ref{vlasovpo})   admits global classical solution if the initial data has radial symmetry and also has compact support in both $x$ and $v$.

With another assumption on the angular momentum,  the spherical symmetry assumption imposed on the initial data $f_0(x,v)$ of the RVP (\ref{vlasovpo}) can be relaxed to the cylindrical symmetry.  The  result of  Glassey-Schaeffer  \cite{glassey2}  says that the   RVP system (\ref{vlasovpo}) admits  global solution if the cylindrically symmetric initial data has compact support  in both $x$ and $v$ and its angular momentum is bounded away from zero, c.f., (\ref{boundedangularmomentum}).  Without loss of generality, we assume that $f_0(x,v)$ is cylindrically symmetric in $x_1x_2$-plane, which will be called the \textit{horizontal plane} in later context. More precisely, the following equality holds for any $x,v\in \R^3$, 
\be\label{april25eqn1}
f_0(x_1,x_2,x_3,v_1,v_2,v_3)= f_0(|\slashed{x}|,0,x_3,|\slashed{v}|,0,v_3),\quad  \slashed{x}:=(x_1,x_2), \quad  \slashed{v}=(v_1,v_2). 
\ee
The lower bound assumption on the angular momentum  of particles  can be understood in the following sense,  
\be\label{boundedangularmomentum}
(\textup{Angular momentum assumption})\qquad f_0(x,v)=0,\quad  \textup{if\,\,} (x,v)\in \{(x,v): x,v\in \R^3, |\slashed x\times \slashed v |\leq  C \},
\ee
where $C$ is some absolute constant.

Due to the gradient structure of the electric field and the cylindrical symmetry of solution, a crucial advantage of imposing the lower bound angular momentum assumption is that the space characteristics are far away from the $z$-axis because the angular momentum $\slashed X(s)\times \slashed V(s)$ is conserved along the characteristics over time. 

More precisely, from the cylindrical symmetry of solution, for simplicity of notation, we define $\tilde{\phi}: \R_t\times \R_{+}\times \R\longrightarrow \R$ as follows, 
\be\label{april26eqn1}
\tilde{\phi}(t,r,x_3):=\phi(t,r, 0,x_3),\quad \nabla_x\phi= \frac{(x_1,x_2,0)}{r} \p_r \tilde{\phi}(t,r,x_3) +(0,0,1)\p_{x_3}\tilde{\phi}(r,x_3), \quad r:=|\slashed x|. 
\ee

Hence, from (\ref{april26eqn1}) and (\ref{characteristicseqn}), we have
\[
\frac{d}{d s } (\slashed X(s;t,x,v)\times \slashed V(s;t,x,v))= \frac{d}{d s } \big( X_1(s;t,x,v) V_2(s;t,x,v)- X_2(s;t,x,v)V_1(s;t,x,v)\big)
\]
\[
= \big( X_1(s;t,x,v) X_2(s;t,x,v)- X_2(s;t,x,v)X_1(s;t,x,v)\big)\p_r \tilde{\phi}(s,|\slashed X(s;t,x,v)|,X_3(s;t,x,v))=0.
\]
That is to say, the angular momentum $\slashed X(s)\times \slashed V(s)$ is conserved along the characteristics.

In this paper, we are interested to improve previous results of Glassey-Schaffer for both the radial symmetry case  in \cite{glassey1} and the cylindrical symmetry case in \cite{glassey2}. More precisely, for the radial symmetry case, our main result is stated as follows, 
 \begin{theorem}\label{maintheoremradial}
Assume that the   initial data  $  f_0(x,v)\in H^s(\R_x^3 \times \R_v^3)$, $s\in \mathbb{Z}_{+},s\geq 6$ is radial in the sense that, $\forall R\in SO(3), f_0(  Rx, Rv) = f_0( x,v). $ Moreover,   we assume that    the initial data decays polynomially as $(x,v)\longrightarrow \infty$ in the following sense, 
\be\label{assumptiononinitialdata}
\sum_{\alpha \in \mathbb{Z}_{+}^6,|\alpha|\leq s} \|(1+|x|+|v|)^{N_r}\nabla_{x,v}^\alpha f_0(x,v)\|_{L^2_{x,v}}< +\infty, \quad N_r:= 100. 
\ee
Then the RVP system \textup{(\ref{vlasovpo})}   admits global solution in $ H^s(\R_x^3 \times \R_v^3)$. 

\end{theorem} 

For the cylindrical symmetry case, our main result is stated as follows, 
\begin{theorem}\label{maintheorem}
Let   $N_c:= 10^{8}. $  If the  cylindrically  symmetric initial data \textup{(}in the sense of \textup{(\ref{april25eqn1})}\textup{)} $  f_0(x,v)\in H^s(\R_x^3 \times \R_v^3)$, $s\in \mathbb{Z}_{+},s\geq 6$      satisfies  the following estimate, 
\be\label{assumcylinrical}
\sum_{\alpha \in \mathbb{Z}_{+}^6,|\alpha|\leq s} \|(1+| x|+|v|)^{N_c}\nabla_{x,v}^\alpha f_0(x,v)\|_{L^2_{x,v}} +\int_{\R^3}\int_{\R^3} \frac{1}{|\slashed x \times \slashed v|^{13}} f_0(x,v) d x d v  < +\infty. 
\ee
Then the RVP  system \textup{(\ref{vlasovpo})} admits global solution in $ H^s(\R_x^3 \times \R_v^3)$. 

\end{theorem} 

A few remarks are in order. 
\begin{remark}
The main merit of this paper are two new observations in the cylindrical symmetry case.

Firstly, to understand the close to  the $z$-axis scenario, we introduce  a  weighted  space-time estimate with a choice of singular weighted function,  see the estimate  (\ref{jan12eqn36}) in  Lemma \ref{spacetimeestimate} for more details.  By using this estimate, we know that particles are not concentrated around  the $z$-axis.

 Lastly,  by  carefully  analyzing  the time resonance set, we show that  a smoothing effect is available for the action of  electric field along characteristics. Roughly speaking, when the frequency is localized away from the time resonance set, $\int_{t_1}^{t_2} P \phi(X(s)) d s$ is smoother than $  P\phi(X(s))$ itself, where $P$ is a Fourier multiplier operator. This type of smoothing effect was pointed out by Klainerman-Staffilani \cite{Klainerman3} in the context of the relativistic Vlasov-Maxwell system. Due to the different speeds of the  electromagnetic field and   massive particles, the  smoothing effect is more intuitive and transparent in the relativistic Vlasov-Maxwell system. However, for  the RVP system (\ref{vlasovpo}),  the smoothing effect is rather obscure and technical. We exploit the smoothing effect by doing normal form transformation after localizing away from the time resonance set.

\end{remark}

\begin{remark}
The compact support assumptions in previous results of Glassey-Schaffer \cite{glassey1,glassey2} are removed   for both the radial case and the cylindrical symmetry case.  Despite that the lower bound assumption of the planar momentum is removed, as stated in (\ref{boundedangularmomentum}),  we still need an vanishing order condition at zero planar momentum for the initial data. The number of order ``$13$'' in  (\ref{assumcylinrical})  can definitely be  improved. However,   there are additional  difficulties in  removing completely  the  vanishing order condition. See also \cite{wang3} for additional remarks and properties of RVP system that might shine some light on the future study. 
\end{remark}
\begin{remark}
 Last and also the least, the plausible  goal of optimizing $s,
$ and $N_r, N_c$   is not pursued here. 
\end{remark}

\subsection{Notation}
 For any two numbers $A$ and $B$, we use  $A\lesssim B$ and $B\gtrsim A$  to denote  $A\leq C B$, where $C$ is an absolute constant. We use the convention that all constants which only depend on the \textbf{initial data}, e.g., the conserved quantities in  (\ref{conservationlaw}),  will be treated as absolute constants.  For any $v=(v_1, v_2,v_3)\in \R^3$ and $u\in \R^3/\{0\}$, we use $\slashed v$ to denotes $(v_1,v_2)\in \R^2$ and use $\tilde{u}$ to denote the direction of $u$, i.e., $\tilde{u}:=u/|u|.$

   We  fix an even smooth function $\tilde{\psi}:\R \rightarrow [0,1]$, which is supported in $[-3/2,3/2]$ and equals to ``$1$'' in $[-5/4, 5/4]$. For any $k,j\in \mathbb{Z}, j> 0$, we define the cutoff functions $\psi_k, \psi_{\leq k}, \psi_{\geq k}:\cup_{n\in \Z_+}\R^n\longrightarrow \R$ as follows, 
\[
\psi_{k}(x) := \tilde{\psi}(|x|/2^k) -\tilde{\psi}(|x|/2^{k-1}), \quad \varphi_0(x):=  \tilde{\psi}(|x|), \quad\forall j\in (0, \infty)\cap \Z, \varphi_j(x):=\psi_j(x) .
\]

\subsection{Local theory and the reduction of the proof}
Because our assumption on the initial data is stronger than the assumption imposed on the distribution function in \cite{luk}, by using the same argument used by Luk-Strain \cite{luk} for the relativistic Vlasov-Maxwell system, which is more difficult, we can reduce the proof of global existence to the  $L^\infty_{x}$-estimate of the electric  field $\nabla_x\phi$, which corresponds to the acceleration of the speed of particles.  

More precisely, assume that $T^{\ast}$ is the maximal time of existence of solution. If  we can show that $\nabla_x\phi\in L^\infty([0,T^{\ast})\times \R_{x}^3)$, then  the lifespan of the solution can be extended to $[0, T^{\ast}+\epsilon]$ for some positive number $\epsilon$. That is to say, the solution of RVP (\ref{vlasovpo})  exists globally in time.

\vo

\noindent \textbf{Plan of this paper}:\quad

\begin{enumerate}

\item[$\bullet$] In section  \ref{radial}, we use  the observation that  particles will travel away from the source of acceleration, i.e., the origin,  to control the majority set in the radial case. 
  
  \item[$\bullet$] In section \ref{spacetime},  we   control   the electric field for a fixed time and prove a weighted space-time estimate for the distribution function, which provides a good control for the electric field near the $z$-axis,  for the cylindrical symmetry case.

\item[$\bullet$] In section \ref{controlmajority},   we show that a smoothing effect is available for the time integral of electric field along characteristics, which enables us to control strongly the majority set, in  the cylindrical symmetry case. 

 \item[$\bullet$]  In section \ref{proofofthe},  thanks to the strong control of the majority set, we show that the high order moment grows at most polynomially for both the radial case and the cylindrical symmetry case. The boundedness of the high order moment   implies the boundedness of electric field in any finite time. Hence finishing the proof of   theorem \ref{maintheoremradial}  and theorem \ref{maintheorem}. 
\end{enumerate}

\vo 

\noindent \textbf{Acknowledgment}\qquad The author is supported by NSFC-11801299, NSFC-12141102, and MOST-2020YFA0713003.

\section{Control of the majority set in the radial   case}\label{radial}

In this section, we main control the velocity characteristics for the radial case. A key observation in   the radial case  is  that, once  the speed of particles reaches a threshold,  the majority of localized particles will travel away from the source of acceleration, which is the origin.  Thanks to the point-wise estimate of the electric field,  we know that  the majority of localized particles will not be accelerated very  much in later time. 

For $n\in \R$, we define
 \be\label{jan12eqn1}
M_{ n}(t,x):=  \int_{\R^3} (1+ |v|)^n f(t,x,v) d v,  \quad M_n(t):= \int_{\R^3}M_{n}(t,x) dx.
 \ee
 From the conservation laws (\ref{conservationlaw}), we know that $M_1(t)$ is  always bounded from the above. Let 
\be\label{enlargedmoment}
   \tilde{M}_n^r (t) := (1+t)^{2n_r}+\sup_{s\in[0,t]} M_{n_r}(s), \quad n_r:=N_r/10. 
\ee
Moreover, we define a set of majorities of particles, which initially localize around zero, at time $s$ as follows,  
\[R^r  (t,s):= \{(X(s;t,x,v),V(s;t,x,v)): | X( 0;t,x,v) |+|V( 0;t,x,v))|\leq  (\tilde{M}^r_n(t))^{1/(2n_r)}  \}.
\]

We have two basic estimates for the   $L^\infty_x$-norm of the electric field $\nabla_x\phi$, which will be elaborated in the next two Lemmas. The first estimate   (\ref{jan30eqn1}) is available mainly because of  the radial symmetry and the conservation law. 

\begin{lemma}
For any $t\in [0, T^{\ast}), x\in \R^3$, the following point-wise estimate holds, 
\be\label{jan30eqn1}
|\nabla_x \phi(t, x)|\lesssim \frac{1}{|x|^2}. 
\ee
\end{lemma}
\begin{proof}
Recall that the distribution $f$ is radial, which implies that the density function $\rho(t,x)$ is radial,
\[
\forall x\in \R^3, \quad \rho(t,Rx)= \int_{\R^3} f(t,Rx, v) d v= \int_{\R^3} f(t,Rx, R \omega) |det(R)| d \omega= \int_{\R^3} f(t, x,\omega) d \omega=\rho(t,x). 
\]
 Hence $\phi(t,x)$ is also radial.  Define
\[
\tilde{\phi}(t,r):= \phi(t,r,0,0), \quad \tilde{\rho}(t, r):=\rho(t,r,0,0),\qquad\Longrightarrow \phi(t,x)=\tilde{\phi}(t, |x|), \rho(t,x)=\tilde{\rho}(t, |x|). 
\]
The Poisson equation for $\phi$ in (\ref{vlasovpo}) is reduced as follows, 
\[
\big(\p_r^2 + \frac{2}{r} \p_r\big) \tilde{\phi}(t,r)= \tilde{\rho}(t,r),\quad
\p_r^2(r\tilde{\phi})(t,r)= r\tilde{\rho}(t, r), \quad \Longrightarrow \p_r( r\tilde{\phi})= \int_{0}^{r} s \tilde{\rho}(t,s)d s +c
\]
\[
\Longrightarrow r\tilde{\phi}= \int_{0}^{r} \int_{0}^{\tau} s \tilde{\rho}(t,s)d s d\tau + cr   = \int_{0}^{r} s(r-s) \tilde{\rho}(t,s) d s +cr   
\]
\[
\Longrightarrow \tilde{\phi}(t,r)= \int_{0}^r s\tilde{\rho}(t,s) d s -  \frac{1}{r} \int_{0}^r s^2 \tilde{\rho}(t,s)d s + c, \quad \Longrightarrow \p_r   \tilde{\phi}(t,r)= \frac{1}{r^2}  \int_{0}^r s^2 \tilde{\rho}(t,s)d s
\]
Hence, from the above equality and the conservation law in (\ref{conservationlaw}), the following estimate holds point-wisely,
\be\label{dec25eqn31}
\Longrightarrow \nabla_x\phi = \frac{x}{|x|}\p_r \tilde{\phi}(t, |x|)=\frac{x}{|x|^3} \int_{0}^{|x|}s^2 \tilde{\rho}(t, s)  d s,\quad \Longrightarrow \nabla_x\phi =\frac{x}{|x|}|\nabla_x\phi |, \quad |\nabla_x\phi |\lesssim \frac{1}{|x|^2}.
\ee
\end{proof} 
 From the   estimate (\ref{jan30eqn1}) in the above Lemma, we know that the acceleration force of particles is weak if the particles are far away from the origin.

In the next Lemma, we show that   the acceleration force of particles is not too strong even if the particles are very close  to the origin. 
\begin{lemma}\label{roughcontrol}
Let   $\delta\in (0,10^{-10})  $ be some fixed sufficiently small constant, then the following estimate holds for any $t\in [0, T^{\ast}), $
\be\label{jan31eqn11}
\|\nabla_x\phi(t,x)\|_{L^\infty_x} \lesssim  1+ \big( \tilde{M}_{n}^r(t) \big)^{  (5+\delta)/  ( (3-\delta)({n_r}-1)  )}.
\ee
\end{lemma}
\begin{proof}
 Note that
 \be\label{dec16eqn11}
\nabla_x\phi(x) = \int_{\R^3} \frac{\rho(t,y)(y-x)}{|x-y|^3} d y= \int_{|y-x|\leq \delta} \frac{\rho(t,y)(y-x)}{|x-y|^3} d y+ \int_{ |y-x|\geq \delta} \frac{\rho(t,y)(y-x)}{|x-y|^3} d y. 
\ee
From the conservation law in (\ref{conservationlaw}), the first part of the above equation is controlled as follows,
\be\label{dec16eqn2}
 \big| \int_{ |y-x|\geq \kappa } \frac{\rho(t,y)(y-x)}{|x-y|^3} d y\big|  \lesssim \frac{1}{\kappa^2}. 
\ee
 For the second part, we use the H\"older inequality by choosing  $ {p}=  ({3-\delta})/{2}$ and $  q =   ({3-\delta})/(1-\delta)$. As a result, we have 
\be\label{dec16eqn13}
\big|\int_{|y-x|\leq \kappa} \frac{\rho(t,y)(y-x)}{|x-y|^3} d y\big|\lesssim \big( \int_{|y-x|\leq \kappa}  \frac{1}{|y-x|^{2p}} dy   \big)^{1/p} \big(\int_{\R^n}  \big(\rho(t,y)\big)^{q}     dy   \big)^{1/q}\lesssim  \kappa^{2\delta/(3-\delta)}\|\rho(t,x)\|_{L^q}.
\ee

Since$\|f(t,x,v)\|_{L^\infty_{x,v}}$ and $M_1(t)$  is bounded all time from the conservation law in (\ref{conservationlaw}),  the following two estimates hold, 
\[
\rho(t,x)= \int_{\R^3} f(t,x,v) d v  \lesssim R^3 + R^{-6/(1-\delta)} M_{6/(1-\delta)}(t,x) \lesssim \big(M_{6/(1-\delta)}(t, x)\big)^{(1-\delta)/(3-\delta)}.
\]
Note that, for $m< n$, we have
\[
M_{m}(t)  
=\int_{ |v|\leq R} \int (1+ |v|)^{m  }  f(t,x,v) d x d v +  \int_{ |v|\geq R} \int (1+|v|)^{m }  f(t,x,v) d x d v  \]
\[
  \lesssim  R^{m-1} + R^{m-n} M_n(t)   \lesssim \big( M_n(t)\big)^{(m-1)/(n-1)}.
\]
 From the above two estimates, we have
\be\label{dec16eqn16}
\|\rho(t,x)\|_{L^q} \lesssim  \big(M_{n_r}(t) \big)^{  (5+\delta)/  ( (3-\delta)(n_r-1)  )}
\ee
Hence,  from the estimates (\ref{dec16eqn11}--\ref{dec16eqn16}), after letting $\kappa=1$,    we have
\be\label{dec25eqn32}
\|\nabla_x\phi(x)\|_{L^\infty_x} \lesssim  1+ \big(M_{n_r}(t) \big)^{  (5+\delta)/  ( (3-\delta)(n_r-1)  )}.
\ee
Hence finishing the proof of the desired estimate (\ref{jan31eqn11}). 
\end{proof}

Now, we study the evolution of the associated  characteristics of the RVP system (\ref{vlasovpo}). 
  From (\ref{vlasovpo}) and (\ref{characteristicseqn}), as a result of direct computations,  we have 
\be\label{jan12eqn21}
\displaystyle{\frac{d}{d s} | X(s )|   = \frac{X(s )}{|X(s )|}\cdot  \widehat{V}(s )}, \quad 
\displaystyle{\frac{d}{d s} | V(s )|   = \frac{V(s )\cdot  X(s )}{| V( s )|| X(s )|}  |\nabla_x\phi(X(s ))  }|.
\ee
From the above two equations, we can see that the quantity $V(s)\cdot  X(s)$ plays an essential role. As a result of direct computations, we have 
 \be\label{jan13eqn41}
  \frac{d}{d s} \frac{V(s )\cdot  X(s )}{   |{V}(s )|      }=|\hat{V}(s )| +  \frac{|X(s )|}{|V(s )|} \big(1 - \frac{\big(V(s )\cdot  X(s )\big)^2}{|  {V}(s )|^2|{X}(s )|^2}\big)|\nabla_x\phi(X(s ))| \geq 0,
 \ee
 From the above equation, we know that the quantity $V(s )\cdot  X(s )$ is an increasing function with respect to time ``$s$''.  With this observation, in the next Proposition, we control the possible size of  the majority set.

\begin{proposition}\label{majority}
For any $t\in[0,T^{\ast})$, the following relation holds for some sufficiently large absolute constant $C$,
\be\label{jan13eqn11}
R^r(t,t)\subset  B(0,  C(\tilde{M}^r_n(t))^{1/(2{n_r})})\times B(0, C  ( \tilde{M}^r_n(t)   )^{  (5+2\delta)/  ((6-2\delta)({n_r} -1)  ) }).
\ee
\end{proposition}
\begin{proof}
Let $t\in[0,T^{\ast})$ be fixed.  
Note that, from the equation (\ref{jan12eqn21}),  the following rough estimate holds for the length of $X(s;t,x,v  )$,  
\be\label{jan13eqn31}
|X(s;  t,x,v )| \leq  |X(0;  t,x,v )| + |s|\leq 2 (\tilde{M}^r_n(t))^{1/(2{n_r})}, \quad s\in[0,t].
\ee

We define the maximal time such that the  velocity characteristic doesn't exceed   the threshold as follows, 
\[
\tau:=\sup\{s: s\in [0, t],\quad \forall \kappa\in[0,s],\, |V(   \kappa;t, {x},  {v})|\leq \big( \tilde{M}^r_n(t) \big)^{  (5+2\epsilon)/  ((6-2\epsilon)( {n_r}-1)  ) }    \}.
\]
From the continuity of characteristics, we know that $\tau>0$. If $\tau=t$, then  there is nothing left to be proved. It remains to consider the case when $0< \tau< t.$

Note that, from the equation (\ref{jan12eqn21}), we know that $  |V( s;t,x,v)|$ is decreasing if $V(s;t,x,v)\cdot  X(s;t,x,v)< 0$. Hence, at the time $\tau$, we have $V(\tau;t,x,v)\cdot  X(\tau;t,x,v)\geq 0$. Otherwise, it contradicts the definition of the 
maximal time.  From the monotonicity  of $V(s;t,x,v)\cdot  X(s;t,x,v)$, see the equation (\ref{jan13eqn41}), we know that $ V(s;t,x,v)\cdot  X(s;t,x,v)\geq 0$ for $s\in [\tau,t]$, which implies that $|V(t,s,x,v)|\geq |V(t,\tau,x,v)|$ for all $s\in[\tau, t]$ from the equation (\ref{jan12eqn21}). To sum up, $\forall s\in [\tau, t],$ we have
\be\label{jan13eqn43}
  |V(s;t,x,v)|\geq |V(\tau;t,x,v)|= \big( \tilde{M}^r_n(t) \big)^{  (5+2\epsilon)/  ((6-2\epsilon)({n_r}-1)  ) },\quad  V(s;t,x,v)\cdot  X(s;t,x,v)\geq 0.
\ee

Starting from the time $\tau$, from the above estimate and  the equation (\ref{jan13eqn41}), we have the following estimate for any $s\in[\tau, t],$
 \be\label{jan13eqn25}
 \frac{V(s;t,x,v)\cdot  X(s;t,x,v)}{|  {V}(s;t,x,v)|} =   \frac{V(\tau;t,x,v)\cdot  X(\tau;t,x,v)}{|  {V}(\tau;t,x,v)|} + \int_{\tau}^{s}\frac{d}{d s} \frac{V(\kappa;t,x,v)\cdot  X( \kappa;t,x,v)}{|  {V}(\kappa;t,x,v)|} d \kappa 
 \gtrsim s-\tau.
 \ee
 From the  equation   (\ref{jan12eqn21}) and the estimates (\ref{jan13eqn43}) and (\ref{jan13eqn25}), we have
 \[
  | X(s;t,x,v)|^2 -  | X(\tau;t,x,v)|^2  =\int_{\tau}^s \frac{d}{d s} | X(\kappa;t,x,v)|^2 d \kappa\gtrsim \int_{\tau}^s  (\kappa-\tau)d\kappa \gtrsim (s-\tau)^2. 
 \]
 From the above estimate, the estimate (\ref{dec25eqn32}), and the  equation in (\ref{characteristicseqn}), the following estimate holds for any $s\in[\tau, t]$,
 \[
  | V(s;t,x,v)|    \leq    | V (\tau;t,x,v)|+    \int_{\tau}^s   | \frac{d}{d s} V(\kappa;t,x,v)|  d \kappa  
   \lesssim \big( \tilde{M}^r_n(t) \big)^{  (5+2\delta)/  ((6-2\delta)({n_r}-1)  ) } +\int_{\tau}^{s}    |\nabla_x\phi(X(t,\kappa,x,v))  | d\kappa
 \]
 \[
 \lesssim \big( \tilde{M}^r_n(t) \big)^{  (5+2\delta)/  ((6-2\delta)(n-1)  ) } +\big|\int_{\tau}^{\tau+\gamma}\big( \tilde{M}^r_n(t) \big)^{  (5+\delta)/  ((3- \delta)({n_r}-1)  ) } d\kappa \big| + \big|\int_{\tau+\gamma}^s \frac{1}{(\kappa-\tau)^{ 2}} d \kappa\big|
 \]
\be\label{jan13eqn53}
 \lesssim \big( \tilde{M}^r_n(t) \big)^{  (5+2\delta)/  ((6-2\delta)({n_r}-1)  ) },\quad \textup{by letting}\quad \gamma =\big( \tilde{M}^r_n(t) \big)^{ - (5+\delta)/  ((6-2 \delta)({n_r}-1)  ) }. 
 \ee
 To sum up,     our desired conclusion (\ref{jan13eqn11}) holds from  (\ref{jan13eqn31}) and (\ref{jan13eqn53}). 
\end{proof}
\section{A space-time estimate for the distribution function in the cylindrical symmetry case}\label{spacetime}

 Our main goal  in this section is to  show a singular weighted space-time estimate for the distribution function in the cylindrical symmetry case, see Proposition  \ref{spacetimeestimate}.  Thanks to this estimate, we know that particles are not too concentrated near $z$-axis even without a lower bound for the angular momentum. 

As in the radial case, 
we also use the classic moment method, to control   the electric field over time. We define 
 \be\label{may2eqn1}
M _{ n_c }(t ):= \int_{\R^3} \int_{\R^3} (1+ |v|)^{n_c} f(t,x,v) d v d x,\quad    \tilde{M}^{c}_n(t) := (1+t)^{n_c^2}+\sup_{s\in[0,t]} M^{  }_{n_c}(s) , \quad n_c:= N_c/10 . 
 \ee
   Moreover,   for any fixed $t\in[0,T^{\star})$, where  $T^{\star}$ denotes the maximal time of existence, we define
\[
M_t: = \inf\{k: k \in \Z_+, 2^k \geq ( \tilde{M}^{c}_n(t) )^{1/(n_c-1)} \}. 
\]

Firstly, as preparation for proving   Proposition  \ref{spacetimeestimate},  we obtain some basic tools, which are point-wise estimates of the localized electric field. More precisely, for any $k\in \Z, j_2\in\Z_+, j_1\in [0,j_2+1]\cap \Z$, we define the localized electric field as follows, 
\begin{multline}\label{may7eqn32}
E_{k;j_1,j_2}(t,x): = \int_{\R^3} \int_{\R^3} \tilde{K}_k(y) f(t, x-y, v)\varphi_{j_1}(\slashed v ) \varphi_{j_2}(v) d y d v,   \\ 
\tilde{K}_k(y) := \int_{\R^3} e^{iy\cdot \xi } i \xi |\xi|^{-2} \varphi_k(\xi) d \xi,\quad  E(t,x)= \sum_{k, j_2\in \Z_+, j_1\in [0,j_2+1]\cap \Z} E_{k;j_1,j_2}(t,x), 
\end{multline}
where, by doing integration by parts in $\xi$ many times, the following estimate holds for the kernel  $\tilde{K}_k(y)$, 
\be\label{may7eqn51}
|\tilde{K}_k(y)|\lesssim 2^{2k}(1+2^k|y|)^{-N_c^5}.
\ee

To  exploit  the benefits of the cylindrical symmetry and the conservation law, we  have  following rough estimate for the electric field. 

\begin{lemma}\label{roughesti}
Let $\epsilon:=10/n_c$. For any fixed $t\in[0,T^{\star})$, the following rough estimate holds for the localized electric field, 
\be\label{may31eqn1}
\|E_{k;j_1,j_2}(t,x)\|_{L^\infty_x} \lesssim \min\{2^{-k+2j_1+j_2}, 2^{2k-j_2}, 2^{2k-n_c j_2}   \tilde{M}^{c}_n(t)\}, 
\ee
\be\label{june2eqn71}
 \|E_{ }(t,x)\|_{L^\infty_x} \lesssim  2^{5M_t/3+2\epsilon M_t}. 
\ee
Moreover, for any $x\in \R^3$ s.t., $|\slashed x|\neq 0$, we have the following point-wise estimate, 
\be\label{may31eqn2}
|E_{k;j_1,j_2}(t,x)|\lesssim 1+ \min\{ \frac{2^{j_1+\epsilon M_t}}{|\slashed x |^{1/2}}, \frac{2^{k-j_2+\epsilon M_t}}{|\slashed x |} \}.
\ee
\end{lemma}
\begin{proof}
Note that the desired estimate (\ref{may31eqn1}) holds straightforwardly from the estimate of kernel in (\ref{may7eqn51}), the volume of support of $v$,  and the conservation law (\ref{conservationlaw}). After summing up the obtained estimate (\ref{may31eqn1})  with respect to $k,j_1, j_2$, our desired rough estimate (\ref{june2eqn71}) holds for the electric field holds. 

Now, we focus on the proof of the desired estimate (\ref{may31eqn2}). Note that, (\ref{may31eqn2}) holds directly from (\ref{may31eqn1}) if $|\slashed x|\leq 2^{-k+\epsilon M_t/2}$. It would be sufficient to consider the case $|\slashed x|\geq 2^{-k+\epsilon M_t/2}$. From  the estimate of kernel in (\ref{may7eqn51}) and the cylindrical symmetry of solution, we have
\[
|E_{k;j_1,j_2}(t,x)|\lesssim 1+ 2^{2k} \int_{\R^3}\int_{|y|\leq 2^{-k+\epsilon M_t/10}} f(t, x-y, v)  \varphi_{j_1}(\slashed v ) \varphi_{j_2}(v) dy d v 
\]
\be\label{may31eqn3}
\lesssim 1+ 2^{k+\epsilon M_t} \int_{\R^3}\int_{\R^3} \frac{1}{|\slashed x|} f(t, x-y, v) \varphi_{j_1}(\slashed v ) \varphi_{j_2}(v)  dy d v \lesssim 1+ \frac{ 2^{k+\epsilon M_t-j_2}}{|\slashed x|}.
\ee
After optimizing the above estimate with the obtained estimate (\ref{may31eqn1}), our desired estimate (\ref{may31eqn2}) holds. 
\end{proof}

  \begin{proposition}\label{spacetimeestimate}
Let $\epsilon^{\star}:=\epsilon/100.$ For any $t\in [0, T^{\ast})$, s.t., $M_t\gg 1$, the following weighted space-time estimate holds,
 \be\label{jan12eqn36}
 A(t):= \int_0^t\int_{\R^3} \int_{\R^3}  \frac{| \slashed v|^{2+2\epsilon^{\star}}}{|\slashed x|^{1-2\epsilon^{\star}}\langle v \rangle }    f(s,x,v)  d x  d v d s \lesssim  2^{5\epsilon M_t}.
 \ee
 As a by-product, the following $L^1_tL^\infty_x$-type estimate holds for any $t_1, t_2\in[0,t]$, 
 \be\label{june9eqn10}
\int_{t_1}^{t_2} \| E_{k;j_1,j_2}(t,\cdot )\|_{L^\infty_x} d  t \lesssim 2^{2\epsilon M_t} + 2^{k-2j_1+j_2+6\epsilon M_t}. 
 \ee
\end{proposition}
\begin{proof}
Define
\be\label{cutoff}
\phi(x):=\left\{\begin{array}{cc} 
2 &x\in [2^{},\infty)\\ 
 2+  (x-2)^3& x\in [1, 2]\\
x^3 & x\in [0,1)\\
0 & x\in (-\infty, 0]\\ 
\end{array}\right. ,\quad  \phi_l(x):= \phi(2^{-l}x).
\ee
From the above explicit formula of cutoff function, we have
\be\label{jan13eqn27}
\phi'(x)\geq 0, \quad \forall x\in (0, \infty),\big| \frac{x\phi'(x)}{\phi(x)}\big|\lesssim  1, \quad \phi_l'(x):=2^{-l}\phi'(2^{-l}x),\quad  \big| \frac{x\phi_l'(x)}{\phi_l(x)}\big|\lesssim  1.
\ee

For $\mu\in \{+,-\}$, we choose a weight function as follows,
 \be\label{may13eqn3}
\omega_{\mu}(x,v):= \big(\mu   |\slashed v | |\slashed x | \slashed x \cdot \slashed v \phi_{-10M_t}(\mu \frac{\slashed x\cdot \slashed v}{|\slashed x||\slashed v |}) +(\slashed x \times \slashed v )^2\big)^{\epsilon^{\star}} \phi_{ }\big(\mu (\frac{\slashed x\cdot \slashed v}{|\slashed x||\slashed v |}+ \frac{1}{2})\big)   . 
\ee
As a result of direct computation, we have
 \[
 \slashed v \cdot \nabla_{\slashed x} \omega_{\mu}(x,v)= \epsilon^{\star} \big(\mu   |\slashed v | |\slashed x | \slashed x \cdot \slashed v \phi_{-10M_t}(\mu \frac{\slashed x\cdot \slashed v}{|\slashed x||\slashed v |}) +(\slashed x \times \slashed v )^2\big)^{\epsilon^{\star}-1}
 \big[  \mu   |\slashed x |  |\slashed v|^3\big( 1+ \frac{(\slashed x \cdot \slashed v)^2}{|\slashed x |^2|\slashed v|^2}\big)   \phi_{-10M_t}(\mu \frac{\slashed x\cdot \slashed v}{|\slashed x||\slashed v |})  +    |\slashed v | |\slashed x | \slashed x \cdot \slashed v 
 \]
 \be\label{may13eqn1}
\times \phi_{-10M_t}'(\mu \frac{\slashed x\cdot \slashed v}{|\slashed x||\slashed v |})   \frac{ ( \slashed x \times \slashed v)^2}{|\slashed x |^3|\slashed v| } \big] + \big(\mu   |\slashed v | |\slashed x | \slashed x \cdot \slashed v \phi_{-10M_t}(\mu \frac{\slashed x\cdot \slashed v}{|\slashed x||\slashed v |}) +(\slashed x \times \slashed v )^2\big)^{\epsilon^{\star}} \phi'\big(\mu (\frac{\slashed x\cdot \slashed v}{|\slashed x||\slashed v |}+ \frac{1}{2})\big)  \frac{ \mu ( \slashed x \times \slashed v)^2}{|\slashed x |^3|\slashed v| }.
 \ee
 From the above equality, we have
 \be\label{may13eqn2}
  \mu  \slashed v \cdot \nabla_{\slashed x} \omega_{\mu}(x,v)\gtrsim  \frac{ |\slashed x |  |\slashed v|^3  \phi_{-10M_t}(\mu \frac{\slashed x\cdot \slashed v}{|\slashed x||\slashed v |}) }{\big(\mu   |\slashed v | |\slashed x | \slashed x \cdot \slashed v \phi_{-10M_t}(\mu \frac{\slashed x\cdot \slashed v}{|\slashed x||\slashed v |}) +(\slashed x \times \slashed v )^2\big)^{ 1-\epsilon^{\star}}}\gtrsim  \frac{|\slashed v|^{1+2\epsilon^{\star}}}{|\slashed x|^{1-2\epsilon^{\star}}}\phi_{-10M_t}(\mu \frac{\slashed x\cdot \slashed v}{|\slashed x||\slashed v |}) \gtrsim  0. 
 \ee
Let
\be\label{may23eqn3}
I_{\mu}(t):= \int_{\R^3} \int_{\R^3}|\slashed v|  \omega_{\mu}(x,v) f(t,x,v) d xd v. 
\ee
As a result of direct computation, we have
\[
\frac{d }{dt} I_{\mu}(t)= \int_{\R^3} \int_{\R^3}|\slashed v|  \hat{v}\cdot\nabla_x\big( \omega_{\mu}(x,v)\big) f(t,x,v) d xd v +\int_{\R^3} \int_{\R^3} E(t,x)\cdot\nabla_v\big(|\slashed v|  \omega_{\mu}(x,v)\big) f(t,x,v) d xd v.
\]
From the above equality and the estimate (\ref{may13eqn2}), we have
\[
 \int_{0}^t \int_{\R^3} \int_{\R^3} \frac{|\slashed v|^{2+2\epsilon^{\star}}}{|\slashed x|^{1-2\epsilon^{\star}}\langle v \rangle }\phi_{-10M_t}(\mu \frac{\slashed x\cdot \slashed v}{|\slashed x||\slashed v |}) f(s,x,v) d  x d v d s 
\]
\be\label{may13eqn71}
\lesssim |I_\mu(t) | + |I_\mu(0) | + \big| \int_{0}^t\int_{\R^3} \int_{\R^3} E(s,x)\cdot\nabla_v\big(|\slashed v|  \omega_{\mu}(x,v)\big) f(s,x,v) d xd v d s \big|.
\ee
Note that, from the estimate (\ref{may13eqn3}) and the conversation law (\ref{conservationlaw}), we have
\be\label{may13eqn72}
\forall s\in [0,t], \quad |I_\mu(s) |\lesssim 1+  \int_{|x|\leq 2^{ M_t/5}} \int_{|v|\leq 2^{2M_t}}|\slashed v|  \omega_{\mu}(x,v) f(t,x,v) d xd v\lesssim 2^{ \epsilon M_t}. 
\ee
From the volume of support, we have
\[
 \int_{0}^t \int_{\R^3} \int_{\R^3} \frac{|\slashed v|^{2+2\epsilon^{\star}}}{|\slashed x|^{1-2\epsilon^{\star}}\langle v \rangle }\big(1- \phi_{-10M_t}( \frac{\slashed x\cdot \slashed v}{|\slashed x||\slashed v |}) -\phi_{-10M_t}( -\frac{\slashed x\cdot \slashed v}{|\slashed x||\slashed v |}) \big) f(s,x,v) d  x d v d s 
\]
\[
\lesssim 1+  \int_{0}^t \int_{|x|\leq 2^{ M_t/5}} \int_{|v|\leq 2^{2M_t}}  \frac{|\slashed v|^{2+2\epsilon^{\star}}}{|\slashed x|^{1-2\epsilon^{\star}}\langle v \rangle }\big(1- \phi_{-10M_t}( \frac{\slashed x\cdot \slashed v}{|\slashed x||\slashed v |}) -\phi_{-10M_t}( -\frac{\slashed x\cdot \slashed v}{|\slashed x||\slashed v |}) \big) f(s,x,v)d xd v d s\]
\be\label{may13eqn73}
\lesssim 1 + 2^{-10M_t + 8 M_t + 10\epsilon M_t}\lesssim 1. 
\ee

Lastly, we estimate the contribution from the nonlinear effect in (\ref{may13eqn71}). Recall  (\ref{may13eqn3}).  As a result of direct computation,  we have
 \be\label{may13eqn6}
  \big|\nabla_v \omega_{\mu}(x,v)\big|\lesssim  \frac{|\slashed x|^2 |\slashed v |}{|\slashed x |^{2-2\epsilon^{\star}} |\slashed v|^{2-2\epsilon^{\star} }}= \frac{|\slashed x|^{2\epsilon^{\star}} }{  |\slashed v|^{1-2\epsilon^{\star} }}. 
\ee

 After localizing the sizes of $|\slashed v |$ and $|v|$ and localizing the electric field, from the estimate (\ref{may13eqn6}), we have
\[
\big| \int_{0}^t\int_{\R^3} \int_{\R^3} E(s,x)\cdot\nabla_v\big(|\slashed v|  \omega_{\mu}(x,v)\big) f(s,x,v) d xd v d s \big|\lesssim \sum_{k\in \Z, j_2\in \Z_+, j_1\in [0, j_2+2]\cap \Z} H^{j_1,j_2}_{k;j_1',j_2'}(t),
\]
where
\be\label{may13eqn31}
H^{j_1,j_2}_{k;j_1',j_2'}(t): =  \int_{0}^t\int_{\R^3} \int_{\R^3} |\slashed x|^{2\epsilon^{\star} } |\slashed v|^{2\epsilon^{\star}} |E_{k;j_1',j_2'}(s,x)| f(s,x,v) \psi_{j_1}(\slashed v) \psi_{j_2}(v) d x d v d  s. 
\ee
 By using the volume of support of $u$ and the estimate of   the kernel  $\tilde{K}_k(y)$ in (\ref{may7eqn51}), we have
 \be\label{may13eqn9}
|E_{k;j_1',j_2'}(s,x)|\lesssim \min\{2^{-k+2j_1'+j_2'}, 2^{2k-j_2'}\}.
\ee

From the above estimate, we can rule out the case when $k\geq  4M_t$ and the case $j_2'\geq 2M_t$ or $j_2\geq 2M_t$. Since there are at most $M_t^3$ cases left, it would be sufficient to let $k, j_1,j_2,j_1',j_2'$ all be fixed. 

Based on the relative size of $j_2$ and $j_2'$, we separate into two cases as follows.

\noindent $\bullet$\qquad If $j_2'\leq j_2$.

Note that, from the cylindrical symmetry and  the estimate of   the kernel  $\tilde{K}_k(y)$ in (\ref{may7eqn51}), we have
\[
|E_{k;j_1',j_2'}(s,x)|\lesssim 2^{\epsilon M_t} + \int_{\R^3}\int_{|y|\leq 2^{-k+\epsilon M_t/10}}  \min\{\frac{2^{k +\epsilon M_t/2}}{|\slashed x-y |}, 2^{2k}\}    f(s,x-y, u)  \varphi_{j_1'}(\slashed u ) \varphi_{j_2'}(u) d y d u. 
\]
Therefore, from the above estimate and the definition of $A(t)$ in (\ref{jan12eqn36}),  we have
\be\label{may13eqn32}
\int_0^t \|E_{k;j_1',j_2'}(s,x)\|_{L^\infty_x} ds  \lesssim  2^{2\epsilon M_t} +2^{k+\epsilon M_t -2j_1'+j_2'} A(t). 
\ee
From the above $L^1_t L^\infty_x$-type estimate and the $L^1_t L^\infty_x$-$L^\infty_t L^1_x$ type bilinear estimate, the following estimate holds if $k\leq 2j_1'-2\epsilon M_t$, 
 \be\label{may13eqn8}
|H^{j_1,j_2}_{k;j_1',j_2'}(t)|\lesssim  2^{3\epsilon M_t} +2^{k+\epsilon M_t-2j_1'+j_2'-j_2} A(t)\lesssim 2^{3\epsilon M_t} + 2^{-\epsilon M_t} A(t). 
\ee
If $k\geq 2j_1'-2\epsilon M_t$, then from the rough estimate of the localized electric field in (\ref{may13eqn9}), we have
 \be\label{may13eqn10}
|H^{j_1,j_2}_{k;j_1',j_2'}(t)|\lesssim 2^{-k+2j_1'+j_2'+ 3\epsilon M_t} 2^{-j_2}\lesssim 2^{5\epsilon M_t}
\ee
To sum up, in whichever case, from the estimates (\ref{may13eqn8}) and (\ref{may13eqn10}), we have
\[
|H^{j_1,j_2}_{k;j_1',j_2'}(t)|\lesssim  2^{5\epsilon M_t} + 2^{-\epsilon M_t} A(t). 
\]
 
 \noindent $\bullet$\qquad If $j_2'\geq j_2$.

Recall (\ref{may13eqn31}). Note that, after changing coordinates $x\longrightarrow x+y$,  we have
\[
\big|H^{j_1,j_2}_{k;j_1',j_2'}(t)\big|\lesssim 2^{5\epsilon M_t}  \int_{0}^t\int_{\R^3} \int_{\R^3} K_k(y)  f(s,x-y, u) f(s,x,v) \varphi_{j_1'}(\slashed u) \varphi_{j_2'}(u) \varphi_{j_1}(\slashed v) \varphi_{j_2}(v) dy  d x d u d v d  s 
\]
\[
\lesssim 2^{5\epsilon M_t}  \int_{0}^t\int_{\R^3} \int_{\R^3}    f(s,x, u) \widetilde{E}_{k;j_2,j_2}(s,x) \varphi_{j_1'}(\slashed u) \varphi_{j_2'}(u)   d x d u  d  s,
\]
where 
\[
\widetilde{E}_{k;j_2,j_2}(s,x):=  \int_{\R^3} \int_{\R^3} K_k(y)f(s,x+y,v)\varphi_{j_1}(\slashed v) \varphi_{j_2}(v) dy d v.
\]

Similar to the obtained estimates (\ref{may13eqn32}) and  (\ref{may13eqn8}), the following estimate holds if $k\leq 2j_1-2\epsilon M_t$,
\[
|H^{j_1,j_2}_{k;j_1',j_2'}(t)|\lesssim  2^{3\epsilon M_t} +2^{k+\epsilon M_t-2j_1+j_2-j_2'} A(t)\lesssim 2^{3\epsilon M_t} + 2^{-\epsilon M_t} A(t).
\]
Note that, from the estimate of   the kernel  $\tilde{K}_k(y)$ in (\ref{may7eqn51})  and the conservation law (\ref{conservationlaw}), we have
\[
|\widetilde{E}_{k;j_2,j_2}(s,x)|\lesssim 2^{-k+2j_1+j_2}. 
\]
From the above estimate, we have the following estimate holds if $k\geq 2j_1-2\epsilon M_t$,
\[
|H^{j_1,j_2}_{k;j_1',j_2'}(t)|\lesssim 2^{-k+2j_1+j_2+3\epsilon M_t-j_2'}\lesssim 2^{5\epsilon M_t}.
\]
To sum up, we have
\[
\big| \int_{0}^t\int_{\R^3} \int_{\R^3} E(s,x)\cdot\nabla_v\big(|\slashed v|  \omega_{\mu}(x,v)\big) f(s,x,v) d xd v d s \big|\lesssim 2^{5\epsilon M_t} + 2^{- \epsilon M_t} A(t). 
\]
Combining the above estimate with the estimates (\ref{may13eqn71}--\ref{may13eqn73}), we have
\[
A(t)\lesssim 2^{5\epsilon M_t} + 2^{- \epsilon M_t} A(t) , \quad \Longrightarrow \qquad  A(t)\lesssim 2^{5\epsilon M_t}.
\]
Hence finishing the proof of our desired estimate (\ref{jan12eqn36}).   The desired estimate (\ref{june9eqn10}) holds directly from the obtained estimates (\ref{jan12eqn36}) and (\ref{may13eqn32}). 
\end{proof}

\section{Control of the majority set in the cylindrical symmetry case}\label{controlmajority}

    Now,  we define    a majority set of particles in  the cylindrical symmetry case at time $s$ as follows,  
\be\label{may10majority}
R^{cyl}  (t,s):= \{(X(s;t,x,v),V(s;t,x,v)): | X( 0;t,x,v) |+|V( 0;t,x,v))|\leq  2^{ M_t/2} \}.
\ee
Moreover, to better capture the size of    velocity characteristics  $|  V(s;t,x,v)|$, we define 
\be\label{may9en21}
 \beta^{}_t(x,v):=\sup_{s\in[0,t] } \inf\{k\in \mathbb{R}_{+}, | V(s;t,x,v)|\leq 2^{k M_t^{   }}  \}, \quad \beta_t:=\sup\{\beta^{ }_t(x,v): (x,v)\in R^{cyl}(t,t)   \}.
\ee

 Based on the estimates of electric field obtained in the previous  section, we aim to proving the proposition for the majority set defined in (\ref{may10majority}), which is the core of the proof of Theorem \ref{maintheorem}.
 \begin{proposition}\label{majorityset}
Let $\epsilon :=10/n_c. $ For any $t\in [0,T^{\star})$, s.t., $M_t\gg 1,$ we  have
\be\label{may11eqn61}
 \beta_t\leq 1-\epsilon/2.    
\ee
\end{proposition}
\begin{proof}

We aim to prove the following estimate 
\be\label{may11eqn11}
 \sup_{s\in [0, t], } |\  V(s)|\leq 2^{ (1-\epsilon)   \beta^{\dagger}  M_t+1}, \quad \beta^{\dagger}:=\max\{\beta_t, 1\}. 
\ee
Define $\zeta$ be the first time that the speed of particle  reaches a threshold as follows 
\be\label{may11eqn12}
\zeta:=\sup \{ s\in [0, t]: \forall \kappa\in [0, s],\, | V(\kappa)|\leq 2^{ (1-\epsilon)  \beta^{\dagger} M_t} \}.
\ee
From the continuity of the speed characteristics, we know that $\zeta\in (0, t]$. Since our desired   estimate (\ref{may11eqn11}) holds straightforwardly if $\zeta=t$, we only have to consider the case $\zeta\in (0, t)$.  Let
\be\label{may11eqn13}
\zeta^{\star}:=\sup\{ s: s\in[\zeta,t], \forall \kappa\in [\zeta, s],  | V(\kappa)|^{ }\in   2^{ (1-\epsilon)  \beta^{\dagger}  M_t} [99/100, 101/100] \}. 
\ee 
Note that, from (\ref{characteristicseqn}), we have
\[
\frac{d}{ds} |V(s)| = \tilde{V}(s)\cdot E(s, X(s)). 
\]
Recall the decomposition of the electric field in (\ref{may7eqn32})  From the estimate (\ref{june1eqn1}) in Lemma \ref{roughlemma1}, the estimate (\ref{june1eqn21}) in Lemma \ref{roughlemma2}, and the estimate (\ref{june1eqn30}) in Lemma \ref{corelemma1}, we know that for any $s\in [\zeta, t]$, we have
\[
|V(s)|\leq |V(\zeta)| + 2^{(1-2\epsilon)M_t} \leq  2^{ (1-\epsilon)  \beta^{\dagger} M_t}  + 2^{(1-2\epsilon)M_t}\leq (1+2^{-1000}) 2^{ (1-\epsilon)  \beta^{\dagger} M_t}.
\] 
Hence improving the bootstrap assumption, which implies that our desired estimate (\ref{may11eqn11}) is true. Recall the definition of $\beta_t$ in (\ref{may9en21}). We have
\[
2^{\beta_t M_t}\leq  2^{ (1-\epsilon)   \beta^{\dagger}  M_t+1}, \quad \Longrightarrow \beta_t M_t\leq (1-\epsilon) M_t + 10, \quad \Longrightarrow  \beta_t\leq (1-\epsilon/2). 
\]
Hence finishing the proof of our desired estimate (\ref{may11eqn61}).

\end{proof}

\begin{lemma}\label{roughlemma1}
For any  $s_1, s_2\in [\zeta, \zeta^{\star}]\subset [0,t]$, the following estimate holds, 
\be\label{june1eqn1}
\sum_{(k,j_1,j_2)\in \mathcal{B}^c}\big|\int_{s_1}^{s_2} \tilde{V}(s)\cdot  E_{k;j_1, j_2}(s, X(s)) d s \big| \lesssim 2^{(1-2\epsilon )M_t},
\ee
where the index set $\mathcal{B} $ is defined as follows, 
\be\label{june1eqn9}
\mathcal{B}:=\{(k,j_1,j_2): k,j_1,j_2\in \Z_+,  k\in[2j_1-4\epsilon M_t, 2j_1+4\epsilon M_t], j_2\in [(1-5.5\epsilon) M_t,(1+\epsilon)M_t],  j_1\geq 5M_t/8+10\epsilon M_t \}.
\ee

\end{lemma}
\begin{proof}
We split possible scenarios of $(k,j_1,j_2)\in \mathcal{B}^c$ into four cases as follows.

\noindent $\bullet$\qquad If $j_2\geq (1+\epsilon ) M_t $. 

Recall (\ref{may7eqn32}). From the rough  estimate of the localized electric field (\ref{may31eqn1}) in Lemma \ref{roughesti}, we have
\[
\sum_{k\in \Z_+, j_2\geq (1+\epsilon ) M_t } \big|\int_{s_1}^{s_2} \tilde{V}(s)\cdot  E_{k;j_1, j_2}(s, X(s)) d s \big|\lesssim \sum_{k\in \Z_+,j_2\geq (1+\epsilon) M_t } \min\{2^{-k+3j_2}, 2^{2k-nj_2}   \tilde{M}^{c}_n(t)  \}
\]
\be\label{2022jan26eqn1}
\lesssim   \sum_{k\in \Z_+,j_2\geq (1+\epsilon) M_t } \min\{2^{-k+3j_2}, 2^{2k-10j_2} \} \lesssim   \sum_{k\in \Z_+,j_2\geq (1+\epsilon) M_t }  2^{-k/100-j_2/100}\lesssim 1. 
\ee

\noindent $\bullet$\qquad If $k\notin[2j_1- 4\epsilon M_t, 2j_1+4\epsilon M_t]$ and $j_2\leq (1+\epsilon )M_t$.

From the obtained estimate (\ref{may13eqn32}) and the rough estimate of the localized electric field (\ref{may31eqn1}) in Lemma \ref{roughesti}, we have
\[
\sum_{k\in \Z_+,k\notin[2j_1-4\epsilon M_t, 2j_1+4\epsilon M_t] }\big|\int_{s_1}^{s_2} \tilde{V}(s)\cdot  E_{k;j_1, j_2}(s, X(s)) d s  \big|  \]
\be\label{2022jan26eqn2}
 \lesssim \sum_{k\in \Z_+,k\notin[2j_1-4\epsilon M_t, 2j_1+4\epsilon M_t] } 2^{\epsilon M_t} \min\{2^{k-2j_1+j_2},2^{-k+2j_1+j_2} \} 
 \lesssim  2^{j_2-3\epsilon M_t}\lesssim 2^{(1-2\epsilon )M_t}.
\ee
\noindent $\bullet$\qquad If $k\in[2j_1-4\epsilon M_t, 2j_1+4\epsilon M_t]$, $j_2\leq (1-5.5\epsilon) M_t$

From the   estimate (\ref{june9eqn10}) in Proposition \ref{spacetimeestimate}  and the rough estimate of the localized electric field (\ref{may31eqn1}) in Lemma \ref{roughesti}, we have
\[
\big|\int_{s_1}^{s_2} \tilde{V}(s)\cdot  E_{k;j_1, j_2}(s, X(s)) d s \big| \lesssim  \int_{s_1}^{s_2}\| E_{k;j_1, j_2}(s, \cdot)\|_{L^\infty_x} d s \lesssim 2^{5\epsilon M_t} +  \min\{2^{k-2j_1+j_2+6\epsilon M_t},2^{-k+2j_1+j_2+\epsilon M_t} \}\]
\be
\lesssim  2^{j_2+3.5\epsilon M_t }\lesssim 2^{(1-2\epsilon )M_t}.
\ee
\noindent $\bullet$\qquad If $k\in[2j_1-4\epsilon M_t, 2j_1+4\epsilon M_t]$, $j_2\in [ (1-5.5\epsilon) M_t, (1+\epsilon) M_t]$, $j_1\leq 5M_t/8+10\epsilon M_t. $

We first rule out the case $|\slashed X(s)|$ is relatively large. From the rough estimate (\ref{may31eqn2}) in Lemma \ref{roughesti}, we have
\be\label{june1eqn6}
 \big|\int_{s_1}^{s_2}\tilde{V}(s)\cdot  E_{k;j_1, j_2}(s, X(s)) \psi_{\geq 2j_1-2M_t+15\epsilon M_t}(\slashed X(s)) d s \big| \lesssim 2^{10\epsilon M_t} +  2^{k-j_2+2\epsilon M_t} 2^{2M_t-2j_1-15\epsilon M_t}\lesssim  2^{(1-3\epsilon) M_t}.
\ee

Now, it  remains to consider the case when  $|\slashed X(s)| $ is relatively small. 
Let
\[
J(s)=\int_{\R^3}\int_{\R^3} |\slashed x \times \slashed v |^{-13} \psi_{\geq  -100 M_t }(\slashed x\times \slashed v ) f(s,x,v) d x d v. 
\]
From our assumption on the initial data, see (\ref{assumptiononinitialdata}), we know  that $0\leq  J(0)\lesssim 1. $ As a result of direct computation, we have
\be\label{june1eqn11}
\frac{d}{ds} J(s) = \int_{\R^3}\int_{\R^3} |\slashed x \times \slashed v |^{-13} \psi_{\geq  -100M_t }(\slashed x\times \slashed v ) \p_sf(s,x,v) d x d v=0, \quad \Longrightarrow 0\leq J(s)= J(0)\lesssim 1. 
\ee
From the above estimate of $J(s)$, we have
\[
|E_{k;j_1,j_2}(s,X(s))|\psi_{<  2j_1-2M_t+15\epsilon M_t}(\slashed X(s))\]
\[
 \lesssim 1+  2^{2k} \int_{\R^3} \int_{|y|\leq 2^{-k+\epsilon M_t/10}} f(s,X(s)-y,v) \psi_{<  2j_1-2M_t+15\epsilon M_t}(\slashed X(s))\varphi_{j_1}(\slashed v ) \varphi_{j_2}(   v ) d y d v 
\]
\[
\lesssim  1+  2^{2k} \int_{\R^3} \int_{|y|\leq 2^{-k+\epsilon M_t/10}} f(s,X(s)-y,v) \psi_{\geq-100 M_t}( (\slashed X(s)-\slashed y)\times \slashed v ) \psi_{<  2j_1-2M_t+15\epsilon M_t}(\slashed X(s))\varphi_{j_1}(\slashed v ) \varphi_{j_2}(   v ) d y d v
\]
\[
+ 2^{2k} \int_{\R^3} \int_{|y|\leq 2^{-k+\epsilon M_t/10}} f(s,X(s)-y,v) \psi_{ < -100 M_t}( (\slashed X(s)-\slashed y)\times \slashed v ) \psi_{<  2j_1-2M_t+15\epsilon M_t}(\slashed X(s))\varphi_{j_1}(\slashed v ) \varphi_{j_2}(   v ) d y d v
\]
\[
\lesssim 1 + 2^{2k} (2^{-k+\epsilon M_t/10}+ 2^{2j_1-2M_t+15 \epsilon M_t})^{13} 2^{13 j_1} \int_{\R^3} \int_{|y|\leq 2^{-k+\epsilon M_t}} \big|(\slashed X(s) - \slashed y)\times \slashed v  \big|^{-13} 
\]
\[ 
\times \psi_{\geq-100 M_t}( (\slashed X(s)-\slashed y)\times \slashed v )f(s,X(s)-y,v) d y d v  + 2^{2k+3j_2-100M_t +10\epsilon M_t}
\]
\be\label{2022jan26eqn3}
\lesssim 1+ 2^{43j_1 -26 M_t+220\epsilon M_t}\lesssim 2^{(1-10\epsilon)M_t}. 
\ee
Hence   our desired estimate (\ref{june1eqn1}) holds from the above obtained estimates (\ref{2022jan26eqn1}), (\ref{2022jan26eqn2}), (\ref{june1eqn6}), and   (\ref{2022jan26eqn3}).
\end{proof}
\begin{lemma}\label{roughlemma2}
For any  $(k,j_1,j_2)\in \mathcal{B}$, see \textup{(\ref{june1eqn9})}, the following estimate holds for any $s_1, s_2\in [\zeta, \zeta^{\star}]\subset[0,t]$,
\be\label{june1eqn21}
\big|\int_{s_1}^{s_2}\tilde{V}(s)\cdot  E_{k;j_1, j_2}(s, X(s))
\big(\psi_{\geq 2j_1-2M_t+15\epsilon M_t}(\slashed X(s)) + \psi_{\leq (M_t-17j_1)/13-2\epsilon M_t}(\slashed X(s))\big)  d s \big| \lesssim 2^{(1-3\epsilon )M_t}.
\ee
\end{lemma}
\begin{proof}
By using the same argument as in the obtained estimate (\ref{june1eqn6}), we have
\[
 \big|\int_{s_1}^{s_2} \tilde{V}(s)\cdot  E_{k;j_1, j_2}(s, X(s)) \psi_{\geq 2j_1-2M_t+15\epsilon M_t}(\slashed X(t)) d s \big| \lesssim 1 +  2^{k-j_2+2\epsilon M_t} 2^{2M_t-2j_1-15\epsilon M_t}\lesssim  2^{(1-3\epsilon) M_t}.
\]
It remains to consider the case $|\slashed X(t)|\lesssim  2^{ (M_t-17j_1)/13-2\epsilon M_t} $.  From the estimate (\ref{june1eqn11}), we have
\[
|E_{k;j_1,j_2}(s,X(s))| \psi_{\leq (M_t-17j_1)/13-2\epsilon M_t}(\slashed X(s))\]
\[
\lesssim 1+  2^{2k} \int_{\R^3} \int_{|y|\leq 2^{-k+\epsilon M_t/10}} f(s,X(s)-y,v)  \psi_{\leq (M_t-17j_1)/13-2\epsilon M_t}(\slashed X(s)) d y d v 
\]
\[
\lesssim 1 + 2^{2k} (2^{-k+\epsilon M_t/10}+ 2^{(M_t-17j_1)/13-2\epsilon M_t} )^{13} 2^{13 j_1} \int_{\R^3} \int_{|y|\leq 2^{-k+\epsilon M_t}}  |(\slashed x - \slashed y)\times \slashed v   |^{-13}\]
\[
\times \psi_{\geq-100 M_t}( (\slashed X(s)-\slashed y)\times \slashed v )  f(s,X(s)-y,v) d y d v  + 2^{2k+3j_2-100M_t+10\epsilon M_t}\lesssim 2^{(1-5\epsilon)M_t}.
\]
Hence finishing the proof of our desired estimate (\ref{june1eqn21}).

\end{proof}
\begin{lemma}\label{corelemma1}
For any  $(k,j_1,j_2)\in \mathcal{B}$, see \textup{(\ref{june1eqn9})}, the following estimate holds for any $s_1, s_2\in [\zeta, \zeta^{\star}]\subset[0, t]$,
\be\label{june1eqn30}
\big|\int_{s_1}^{s_2} \tilde{V}(s)\cdot  E_{k;j_1, j_2}(s, X(s))  \psi_{[  (M_t-17j_1)/13-2\epsilon M_t, 2j_1-(2 -15\epsilon) M_t ]}(\slashed X(s))  d  s \big| \lesssim 2^{(1-3\epsilon )M_t}.
\ee
\end{lemma}
\begin{proof}
Let 
\[
l_1:= j_1-M_t- 5\epsilon M_t, \quad \alpha = \frac{2}{3}l_1, \quad l_2:= \alpha - 20\epsilon M_t, \quad \theta_{V(s)}(v):=(\hat{V}(s)-\hat{v})/|\hat{V}(s)-\hat{v}|.
\]

 Based on the possible size of  $\theta_{V(s)}(v)\cdot\xi/|\xi| $,  $|\hat{V}(s)-\hat{v}|$, and the angle between   $ (-X_2(s), X_1(s),0)/|\slashed X(s)|$ and $ \theta_{V(s)}(v)$,  we decompose the integral in (\ref{june1eqn30}) into three parts as follows, 
\be\label{june1eqn52}
\int_{s_1}^{s_2} \tilde{V}(s)\cdot  E_{k;j_1, j_2}(s, X(s))  \psi_{[  (M_t-17j_1)/13-2\epsilon M_t, 2j_1-(2-15\epsilon) M_t ]}(\slashed X(s))  d s= \sum_{i=1,2,3} \int_{s_1}^{s_2}   H_{k,j_1,j_2}^i(s ) d s,  
\ee
where
\[
 H_{k,j_1,j_2}^1(s )=  \int_{\R^3} \int_{\R^3 } e^{i X(s)\cdot \xi- i s\hat{v}\cdot \xi} i  \tilde{V}(s)\cdot \xi |\xi|^{-2}  \hat{g}(s ,\xi ,v) \varphi_{j_1}(\slashed v)\varphi_{j_2}(v)  \varphi_k(\xi) \varphi_{l_1,\alpha}(v, X(s), \tilde{V}(s))
\]
\be 
  \times \psi_{[  (M_t-17j_1)/13-2\epsilon M_t, 2j_1- (2-15\epsilon) M_t ]}(\slashed X( s)) d v d \xi ,
  \ee
  \[
   H_{k,j_1,j_2}^2(s )=  \int_{\R^3} \int_{\R^3 } e^{i X(s)\cdot \xi- i s \hat{v}\cdot \xi}   i  \tilde{V}(s)\cdot \xi |\xi|^{-2} \hat{g}(s ,\xi ,v) \varphi_{j_1}(\slashed v)\varphi_{j_2}(v)    \varphi_k(\xi)\psi_{\leq l_2 }(  \theta_{V(s)}(v)\cdot \tilde{\xi} )
  \]
\be 
  \times  \big(1- \varphi_{l_1,\alpha}(v, X(t), \tilde{V}(s))\big)  \psi_{[  (M_t-17j_1)/13-2\epsilon M_t, 2j_1- (2-15\epsilon) M_t ]}(\slashed X(s)) d v d \xi ,
  \ee
  \[
 H_{k,j_1,j_2}^3(s )=  \int_{\R^3} \int_{\R^3 } e^{i X(s)\cdot \xi- i s \hat{v}\cdot \xi}   i  \tilde{V}(s)\cdot \xi |\xi|^{-2} \hat{g}(s ,\xi ,v)\varphi_{j_1}(\slashed v)\varphi_{j_2}(v)    \varphi_k(\xi)  \psi_{> l_2 }(  \theta_{V(s)}(v)\cdot \tilde{\xi} )
  \]
 \be\label{june2eqn1} 
 \times  \big(1- \varphi_{l_1,\alpha}(v, X(s), \tilde{V}(s))\big)  \psi_{[  (M_t-17j_1)/13-2\epsilon M_t, 2j_1- (2-15\epsilon) M_t ]}(\slashed X(s )) d v d \xi.
  \ee
   where the cutoff function $ \varphi_{l_1,\alpha}(v, \tilde{V}(t))$ is defined as follows, 
 \be\label{cutofffunc}
 \varphi_{l_1,\alpha}(v, X(s), \tilde{V}(s)):= \psi_{\leq l_1  }(|\tilde{v}-\tilde{V}(s)|) + \psi_{>l_1  }(|\tilde{v}-\tilde{V}(s)|)\psi_{\leq \alpha} (\theta_{V(s)}(v)\times (-X_2(s), X_1(s),0)/|\slashed X(s)|). 
 \ee

\noindent $\bullet$\qquad The estimate of $ H_{k,j_1,j_2}^1(s )$. \qquad Note that, in terms of kernel, we have
 \[
 H_{k,j_1,j_2}^1 (s )=  \int_{\R^3} \int_{\R^3} \tilde{V}(s)\cdot \tilde{K}_k(y) f(s,x-y, v) \varphi_{j_1}(\slashed v)\varphi_{j_2}(v) \varphi_{l_1,\alpha}(v, X(s), \tilde{V}(s)) d v d y,
 \]
 where the kernel $\tilde{K}_k(y)$ was defined in (\ref{may7eqn32}).   Note that, for any fixed  $a, b\in \mathbb{S}^2, a\neq b,$ $0< \epsilon \ll 1$, we have
 \be\label{april29eqn1}
S_{a,b}:=\{c: c\in \mathbb{S}^2, |(a-c)\times b|\leq \epsilon\}, \quad |S_{a,b}|\lesssim \epsilon^{3/2}. 
\ee
 The above claim follows from the following argument. Let 
 $
  \mathbf{e}_1:=b,\quad \mathbf{e}_2:= ({a - (b\cdot a) b})/{|a-(b\cdot a) b |},  \quad \mathbf{e}_3= \mathbf{e}_1\times \mathbf{e}_2.$ Then $\{\mathbf{e}_1, \mathbf{e}_2, \mathbf{e}_3\}$ is a orthonormal frame. In terms of the orthonormal frame,  we have
 \begin{multline}\label{june1eqn60}
 a= a_1 \mathbf{e}_1 + a_2 \mathbf{e}_2, \quad c=\cos\theta\cos \phi \mathbf{e}_1+ \cos\theta \sin  \phi \mathbf{e}_2+\sin\theta \mathbf{e}_3,\quad  \theta\in [-\pi/2, \pi/2], \phi \in [0, 2\pi],  \\ 
(a-c)\times b  = -(a_2-\cos\theta \sin  \phi) \mathbf{e}_3 + \sin\theta \mathbf{e}_2, \\ 
 |(a-c)\times b|\leq \epsilon \quad \Longrightarrow |\sin \theta  |\leq \epsilon, \quad |a_2-\sin \phi|\lesssim \epsilon, \quad |a-c|\lesssim |a_1-\cos\phi|+|a_2-\sin\phi| + \epsilon \lesssim \epsilon^{1/2},\\ 
   \Longrightarrow \big|\{(\theta, \phi): \theta\in [-\pi/2, \pi/2], \phi\in [0, 2\pi],  |\sin \theta  |\leq \epsilon, \quad |a_2-\sin \phi|\lesssim \epsilon  \}\big|\lesssim \epsilon^{3/2}.
 \end{multline}

If $v\in  supp_v (\psi_{\leq \alpha} (\theta_{V(s)}(v)\times (-X_2(s), X_1(s),0)/|\slashed X(s)|) \psi_{n}(|\tilde{v}-\tilde{V}(s)|))$, $n\in[l_1, 2]\cap \mathbb{Z}$, after letting $\epsilon=2^{n+\alpha}$,  the following estimate holds from    estimates in (\ref{june1eqn60}), 
\[
Vol\big(supp_v(\psi_{\leq \alpha} (\theta_{V(s)}(v)\times (-X_2(s), X_1(s),0)/|\slashed X(s)|) \psi_{n}(|\tilde{v}-\tilde{V}(s)|)\varphi_{j_2}(v))\big) \lesssim 2^{3j_2}2^{3(n+\alpha)/2},\]
\[
|\hat{v}-\tilde{v}|\leq \langle v \rangle^{-2},\quad \Longrightarrow  2^{n-2} \leq |\hat{V}(s)-\hat{v}|+ 2^{-2M_t + 3\epsilon M_t} \lesssim 2^{(n+\alpha)/2}+2^{-2M_t + 3\epsilon M_t},
\]
\be
   \Longrightarrow n\leq \max\{\alpha, -2 M_t +2\epsilon M_t \}+C=\alpha+C,
\ee
where $C$ is some absolutely constant.

 From the above estimate, the estimate of the kernel  $\tilde{K}_k(y)$ in (\ref{may7eqn51}), and the volume of support of $v$, we have
\be\label{june2eqn61}
| H_{k,j_1,j_2}^1 (s )|\lesssim 2^{\epsilon M_t}\big( 2^{-k+3j_2+ 2l_1} + 2^{-k+3j_2+ 3\alpha}\big)\lesssim 2^{(1-3\epsilon)M_t}. 
\ee

\noindent $\bullet$\qquad The estimate of $ H_{k,j_1,j_2}^2(s )$. \qquad Note that, in terms of kernel, we have
\[
H_{k,j_1,j_2}^2(s) = \int_{\R^3} \int_{\R^3} f(s,X(s)-y, v) \tilde{K}_{k,l_2}(y, v,V(s))  \big(1- \varphi_{l_1,\alpha}(v, X(s), \tilde{V}(s))\big)
\]
 \be\label{april28eqn11}
\times \psi_{[  (M_t-17j_1)/13-2\epsilon M_t, 2j_1- (2-15\epsilon) M_t ]}(\slashed X(s)) \varphi_{j_1}(\slashed v)\varphi_{j_2}(v) d v d y, 
 \ee
 where
 \[
 \tilde{K}_{k,l_2}(y, v,V(s)) = \int_{\R^3 } e^{i y\cdot \xi} i \xi |\xi|^{-2} \varphi_k(\xi)   \psi_{\leq l_2 }(  \theta_{V(s)}(v)\cdot \tilde{\xi} )  d \xi.  
 \]
 By doing integration by parts in $\theta_{V(s)}(v)$ direction and $(\theta_{V(s)}(v))^{\bot}$ directions, the following estimate holds for the kernel $ \tilde{K}_{k,l_2}(y, v,V(s))$,
 \be\label{april28eqn1}
 | \tilde{K}_{k,l_2}(y, v,V(s))| \lesssim 2^{2k+ l_2}(1+ 2^{k+l_2}|y\cdot \theta_{V(s)}(v)|)^{-N_c^3} (1+2^k|y\times \theta_{V(s)}(v)|)^{-N_c^3}
 \ee
 From the above estimate, we know that ``$y$'' is localized inside a cylinder with base in the plane perpendicular to $\theta_{V(s)}(v)^{}$. Due to the cutoff function $ (1- \varphi_{l_1,\alpha}(v, X(s), \tilde{V}(s)) )  $ in (\ref{april28eqn11}),  the angle between $\theta_{V(s)}(v)$ and $ (-X_2(s), X_1(s),0)/|\slashed X(s)|)$ is greater than $2^{\alpha}$, which means that the intersection of the cylinder with any $x_1x_2$ plane  is less than $(2^{-k-\alpha})^2.$

 Moreover, note that, for the case we are considering, we have
 \be\label{2022jan26eqn41}
 |\slashed X(s)|\geq 2^{(M_t-17j_1)/13-2\epsilon M_t} \geq 2^{-2j_1-l_2 + 10\epsilon M_t}\geq 2^{-k-l_2+5\epsilon M_t}. 
 \ee
 Therefore, from the cylindrical symmetry of solution and the estimate (\ref{jan12eqn36}) in Proposition \ref{spacetimeestimate}, we have
\[
\int_{t_1}^{t_2}\big| H_{k,j_1,j_2}^2(s)\big| d s \lesssim  2^{2k+l_2} \int_{t_1}^{t_2} \int_{\R^3} \int_{\R^3} (1+ 2^{k+l_2}|y\cdot \theta_{V(s)}(v)|)^{-N_0^3} (1+2^k|y\times \theta_{V(s)}(v)|)^{-N_0^3} f(s, X(s)-y, v) \varphi_{j_1}(\slashed v)
 \]
\[ 
 \times\varphi_{j_2}(v) \big(1- \varphi_{l_1,\alpha}(v, X(s), \tilde{V}(s))\big)  \psi_{[  (M_t-17j_1)/13-2\epsilon M_t, 2j_1- (2-15\epsilon) M_t ]}(\slashed X(s ))  d v d y  ds   \]
 \[\lesssim 1 + \int_{t_1}^{t_2} \int_{\R^3} \int_{|y|\leq 2^{-k-l_2+\epsilon M_t/2}}  \frac{2^{2k+l_2-k-\alpha +\epsilon M_t/10}}{|\slashed X(t)- \slashed y|} f(s, X(s)-y, v)  \psi_{[  (M_t-17j_1)/13-\epsilon M_t, 2j_1- (2-15\epsilon) M_t ]}(\slashed X(s ))  
\] 
 \be\label{june2eqn11}
  \times \varphi_{j_1}(\slashed v)\varphi_{j_2}(v) d y d v  d s   \lesssim 1+ 2^{2k+l_2-k-\alpha +6\epsilon M_t } 2^{-2j_1+ j_2}   \lesssim 2^{(1-5\epsilon) M_t}. 
\ee


\noindent $\bullet$\qquad The estimate of $ H_{k,j_1,j_2}^3(s )$. \qquad

Recall (\ref{june2eqn1}).   Note that, for any $(v,\xi)\in supp( \psi_{> l_2 }(  \theta_{V(s)}(v)\cdot \tilde{\xi} ) (1- \varphi_{l_1,\alpha}(v, X(s), \tilde{V}(s))) )$, we have
 \be\label{april28eqn13}
|\hat{V}(s)-\hat{v}|\gtrsim |\tilde{V}(s)-\tilde{v}|-  2^{-2M_t + 4\epsilon M_t}\gtrsim 2^{l_1}, \quad \Longrightarrow 
|\hat{V}(t)\cdot\tilde{\xi} - \hat{v}\cdot \tilde{\xi}|= |\hat{V}(t)-\hat{v}||\theta_{V(t)}(v)\cdot \tilde{\xi}|\gtrsim 2^{l_1+l_2}. 
\ee

Let $g(s,x,v):= f(s, x+ s\hat{v}, v).$ For $H_{k,j_1,j_2}^3(s)$, we do integration by parts in time once. As a result, we have
\be\label{june2eqn54}
\int_{s_1}^{s_2} H_{k,j_1,j_2}^3(s)d s = End_{k,j_1,j_2}(s_1, s_2) + \widetilde{H}_{k,j_1,j_2}^1(s_1,s_2) +  \widetilde{H}_{k,j_1,j_2}^2(s_1,s_2),
\ee
where
\[
 End_{k,j_1,j_2}(s_1, s_2):= \sum_{a=1,2} \int_{\R^3} \int_{\R^3 } e^{i X(s_a)\cdot \xi- i  s_a \hat{v}\cdot \xi} i \tilde{V}(s)\cdot \xi |\xi|^{-2}  \hat{g}(s_a ,\xi ,v)     \varphi_k(\xi)   \varphi_{j_1}(\slashed v)\varphi_{j_2}(v) (\hat{V}(s_a)\cdot\xi - \hat{v}\cdot \xi )^{-1} 
\]
\[
\times   \psi_{> l_2 }(  \theta_{V(s_a)}(v)\cdot \tilde{\xi} )  \big(1-  \varphi_{l_1,\alpha}(v, X(s_a), \tilde{V}(s_a))\big)  \psi_{[  (M_t-17j_1)/13-2\epsilon M_t, 2j_1- (2-15\epsilon) M_t ]}(\slashed X(s_a))  d v d \xi
\]
 \be\label{april28eqn14}
=  \sum_{a=1,2}  \sum_{ n\in[l_1,2]\cap \mathbb{Z}}    \sum_{ l\in[l_2,2]\cap \mathbb{Z}} \int_{\R^3} \int_{\R^3} f(s_a, X(s_a)-y, v) \widetilde{K}^0_{k,n,l }(y, X(s_a),  V(s_a), v)  \varphi_{j_1}(\slashed v)\varphi_{j_2}(v) d y d v, 
\ee
\[
 \widetilde{H}_{k,j_1,j_2}^1(s_1, s_2):=\int_{s_1}^{s_2}  \int_{\R^3} \int_{\R^3 } e^{i X(s)\cdot \xi- i s \hat{v}\cdot \xi}    i\tilde{V}(s)\cdot \xi |\xi|^{-2} \varphi_k(\xi)  \varphi_{j_1}(\slashed v)\varphi_{j_2}(v)  \p_s  \hat{g}(s ,\xi ,v)   \big(1-  \varphi_{l_1,\alpha}(v, X(s ), \tilde{V}( s))\big) 
  \]
 \be\label{april28eqn15}
  \times   \psi_{[  (M_t -17j_1)/13-2\epsilon M_t, 2j_1- (2-15\epsilon) M_t]}(\slashed X(t ))(\hat{V}(s)\cdot\xi - \hat{v}\cdot \xi )^{-1}   \psi_{> l_2 }(  \theta_{V(s)}(v)\cdot \tilde{\xi} ) d v d \xi d s,  
\ee
\[
 \widetilde{H}_{k,j_1,j_2}^2(s_1, s_2):=\int_{s_1}^{s_2}  \int_{\R^3} \int_{\R^3 } e^{i X(s)\cdot \xi }  \p_s \big( i \tilde{V}(s)\cdot \xi |\xi|^{-2}  \psi_{> l_2 }(  \theta_{V(s)}(v)\cdot \tilde{\xi} )
  \psi_{[  ( M_t-17j_1)/13-2\epsilon M_t, 2j_1- (2-15\epsilon) M_t ]}(\slashed X(t ))      \]
  \[
  \times (\hat{V}(s)\cdot\xi - \hat{v}\cdot \xi )^{-1} \big(1-  \varphi_{l_1,\alpha}(v, X(s), \tilde{V}(s))\big) \big)   \varphi_k(\xi)   \varphi_{j_1}(\slashed v)\varphi_{j_2}(v) \hat{f}(s ,\xi ,v) d v d \xi d s 
\]
 \be\label{april28eqn16}
  =  \sum_{ n\in[l_1,2]\cap \mathbb{Z}}    \sum_{ l\in[l_2,2]\cap \mathbb{Z}} \int_{s_1}^{s_2} \int_{\R^3} \int_{\R^3} f(s, X(s)-y, v) \widetilde{K}^1_{k,n,l}(y,X(s),  V(s), v)   \varphi_{j_1}(\slashed v)\varphi_{j_2}(v) d y d v d s,
\ee
where the kernels $\widetilde{K}^i_{k,l_1,l_2}(y, V(s), v), i\in\{0,1\},$ are defined as follow,
\[
\widetilde{K}^0_{k,n,l   }(y, X(s_a),  V(s_a), v):= \int_{\R^3} e^{i y \cdot \xi} i\tilde{V}(s_a)\cdot  \xi |\xi|^{-2}\varphi_k(\xi) (\hat{V}(s_a)\cdot\xi - \hat{v}\cdot \xi )^{-1}   \psi_{  l  }(  \theta_{V(s_a)}(v)\cdot \tilde{\xi} )  \psi_{n}(|\tilde{v}-\tilde{V}(s_a)|) 
\]
\be\label{april28eqn21}
 \times   \psi_{[  ( M_t-17j_1)/13-2\epsilon M_t, 2j_1- (2-15\epsilon) M_t ]}(\slashed X(s_a ))   \big(1-  \varphi_{l_1,\alpha}(v, X(s_a ), \tilde{V}( s_a))\big)d \xi, 
\ee
\[
\widetilde{K}^1_{k, n,l  }(y, X(s), V(s), v):= \int_{\R^3} e^{i y \cdot \xi} \p_s \big( i \tilde{V}(s)\cdot \xi |\xi|^{-2}  \psi_{> l_2 }(  \theta_{V(s)}(v)\cdot \tilde{\xi} )
  (\hat{V}(s)\cdot\xi - \hat{v}\cdot \xi )^{-1}    \psi_{  l  }(  \theta_{V(s)}(v)\cdot \tilde{\xi} )  
\]
\be\label{april28eqn22}
\times   \psi_{[  ( M_t-17j_1)/13-2\epsilon M_t, 2j_1- (2-15\epsilon) M_t]}(\slashed X(s ))   \psi_{n}(|\tilde{v}-\tilde{V}(s)|)  \big(1-  \varphi_{l_1,\alpha}(v, X(s), \tilde{V}(s))\big) \big) \varphi_k(\xi) d \xi.
\ee

\noindent $\bullet$\quad The estimate of $ End_{k,j_1,j_2}(t_1, t_2)$.

Recall (\ref{april28eqn14}).   Note that,   by using the estimate   (\ref{april28eqn13}) and  doing integration by parts in $\theta_{V(t_a)}(v)$ direction and $(\theta_{V(s_a)}(v))^{\bot}$ directions, the following estimate holds for the kernel $\widetilde{K}^0_{k,l_1,l_2}(y, V(s_a), v)$. 
\[
|\widetilde{K}^0_{k,n,l}(y, X(s_a), V(s_a ), v)| \lesssim  2^{k+ l-l-n}   \psi_{[  (M_t-17j_1)/13-2\epsilon M_t, 2j_1- (2-15\epsilon) M_t ]}(\slashed X(s_a)) 
\]
 \be\label{ma}
\times (1+2^{k+l}|y\cdot\theta_{V(s_a)}(v) |)^{- N_c^3} (1+2^{k }|y\times \theta_{V(s_a )}(v) |)^{- N_c^3 }   . 
\ee 
From the above estimate, the estimate (\ref{2022jan26eqn41}), and the cylindrical symmetry of the distribution function,  we have 
\[
\big|  End_{k,j_1,j_2}(s_1, s_2)\big| \lesssim \sum_{a=1,2} \sum_{ n\in[l_1,2]\cap \mathbb{Z}}    \sum_{ l\in[l_2,2]\cap \mathbb{Z}} \int_{\R^3} \int_{\R^3} f(s_a, X(s_a)-y, v)  2^{k -n} (1+2^{k+l}|y\cdot\theta_{V(s_a)}(v) |)^{- N_0^3} \] 
\[
\times  (1+2^{k }|y\times \theta_{V(s_a )}(v) |)^{- N_0^3 }  \varphi_{j_1}(\slashed v)\varphi_{j_2}(v) \psi_{[  (M_t-17j_1)/13-2\epsilon M_t, 2j_1- (2-15\epsilon) M_t ]}(\slashed X(s_a )) d y d v 
\]
\be\label{june2eqn51}
\lesssim 2^{2\epsilon M_t}+ 2^{ k-l_1-k-\alpha +4\epsilon M_t }2^{-j_2} 2^{-(M_t-17j_1)/13 } \lesssim 2^{(1-10\epsilon)M_t}. 
\ee

\noindent $\bullet$\quad The estimate of $  \widetilde{H}_{k,j_1,j_2}^1 (s_1,s_2)$.

Recall (\ref{april28eqn15}) and  (\ref{vlasovpo}). Note that, 
\[
\p_s g(s,x,v) = \nabla_x\phi(s,x+s\hat{v})\cdot \nabla_v f(s,x+s\hat{v}, v ). 
\]
Hence, after doing integration by parts in $v$, we have 
\[
   \widetilde{H}_{k,j_1,j_2}^1 (s_1,s_2):=\int_{s_1}^{s_2}  \int_{\R^3} \int_{\R^3 }\int_{\R^3}  e^{i X(s)\cdot \xi }   \psi_{[  (M_t-17j_1)/13-2\epsilon M_t, 2j_1- (2-15\epsilon) M_t ]}(\slashed X(s)) \]
   \[
   \times  i \tilde{V}(s)\cdot \xi |\xi|^{-2} \varphi_k(\xi) \widehat{f}(s,\xi-\eta,v)\widehat{\nabla_x \phi}(s, \eta) \cdot \nabla_v \big(  \varphi_{j_1}(\slashed v) \varphi_{j_2}(v)  (1-  \varphi_{l_1,\alpha}(v, X(s ), \tilde{V}( s)))
  \]
  \[
    \times  \psi_{> l_2 }(  \theta_{V(s)}(v)\cdot \tilde{\xi} )   \psi_{ >   l_1 }(|\tilde{v}-\tilde{V}(s)|)  (\hat{V}(s)\cdot\xi - \hat{v}\cdot \xi )^{-1} \big) d v d\eta d \xi d s 
  \]
 \be\label{april28eqn17}
=\sum_{ n\in[l_1,2]\cap \mathbb{Z}}    \sum_{ l\in[l_2,2]\cap \mathbb{Z}}  \int_{s_1}^{s_2}  \int_{\R^3} \int_{\R^3 }\int_{\R^3}    \widetilde{K}^2_{k,n,l}(y, V(s), v)\cdot \nabla_x\phi (s, X(s)-y) f(s,X(s)-y,v)  dy d v d s,
\ee
where the kernel  $    \widetilde{K}^2_{k,n,l}(y, V(s), v)$ is defined as follows, 
 \[
   \widetilde{K}^2_{k,n,l}(y, V(s), v):= \int_{\R^3} e^{i x\cdot \xi }   \nabla_v \big( \varphi_{j_1}(\slashed v)  \varphi_{j_2}(v) \big(1-  \varphi_{l_1,\alpha}(v, X(s ), \tilde{V}( s)) (\hat{V}(s)\cdot\xi - \hat{v}\cdot \xi )^{-1}  \psi_{l }(  \theta_{V(s)}(v)\cdot \tilde{\xi} )
 \]
 \be\label{april29eqn11}
  \times     \psi_{ n }(|\tilde{v}-\tilde{V}(s)|) \big)  i \tilde{V}(s)\cdot \xi |\xi|^{-2} \varphi_k(\xi) \psi_{[  (M_t-17j_1)/13-2\epsilon M_t, 2j_1- (2-15\epsilon) M_t ]}(\slashed X(s )) d \xi.   
\ee

By   doing integration by parts in $\theta_{V(s)}(v)$ direction and $(\theta_{V(s)}(v))^{\bot}$ directions,  as a result of direct computation, we have
 \be\label{may5eqn51}
|\widetilde{K}^2_{k,n,l   }(y, V(s), v)| \lesssim 2^{k+ l}\big( 2^{ -2l-2n-j_2 } + 2^{ -l-n-j_1 } \big)   (1+2^{k+l}|y\cdot\theta_{V(s)}(v) |)^{- N_c^3 } (1+2^{k }|y\times \theta_{V(s)}(v) |)^{-  N_c^3 } .
\ee
Moreover, from the   estimate   (\ref{june9eqn10}) in Proposition \ref{spacetimeestimate}, and the rough estimate (\ref{may31eqn1}) in Lemma \ref{roughesti}, we have
\[
 \int_{s_1}^{s_2}  \|  \nabla_x\phi (s, \cdot ) \|_{L^\infty_x } ds \lesssim \sum_{k, j_2\in \Z_+, j_1\in [0, j_2+2]\cap \Z} \min\{ 2^{-k+2j_1+j_2}, 2^{2k-j_2},  2^{2k-n_c j_2}   \tilde{M}^{c}_n(t), 2^{2\epsilon M_t} + 2^{k-2j_1+j_2+6\epsilon M_t}\}
\]
\be\label{june2eqn31}
\lesssim 2^{(1+4\epsilon) M_t}.  
\ee

From the above estimate,  the estimate  of kernel in (\ref{may5eqn51}),  the estimate (\ref{2022jan26eqn41}), and the cylindrical symmetry of the distribution function, we have 
 \[
\big|\widetilde{H}_{k,j_1,j_2}^1 (s_1,s_2)\big|\lesssim \sum_{ n\in[l_1,2]\cap \mathbb{Z}}    \sum_{ l\in[l_2,2]\cap \mathbb{Z}}   \int_{s_1}^{s_2}  \|  \nabla_x\phi (s, \cdot ) \|_{L^\infty_x }  \frac{2^{ k+l-k-\alpha +\epsilon M_t/10}}{|\slashed X(t)- \slashed y|}\big( 2^{ -2l-2n-j_2 } + 2^{ -l-n-j_1 } \big)  (1+2^{k+l}|y\cdot\theta_{V(s)}(v) |)^{- N_c^3 } \]
\[
\times  (1+2^{k }|y\times \theta_{V(s)}(v) |)^{-  N_c^3 }  \varphi_{j_1}(\slashed v)\varphi_{j_2}(v)  f(t,X(s)-y,v)\psi_{[  ( M_t-17j_1)/13-2\epsilon M_t, 2j_1- (2-15\epsilon) M_t]}(\slashed X(s ))  dy d v d s\]
\[
\lesssim \int_{s_1}^{s_2}   2^{-j_2+4\epsilon M_t}\big( 2^{ -l_1-\alpha-j_1} + 2^{-l_2 -2l_1 -\alpha-j_2 }  \big) 2^{-( M_t-17j_1)/13}    \|  \nabla_x\phi (t, \cdot ) \|_{L^\infty_x }   ds 
\]
\be\label{june2eqn52}
  \lesssim  2^{15\epsilon M_t} \big( 2^{ -l_1-\alpha-j_1} + 2^{-l_2 -2l_1 -\alpha-j_2 }  \big) 2^{-( M_t-17j_1)/13} \lesssim 2^{(1-3\epsilon)M_t}. 
\ee

\vo 

\noindent $\bullet$\quad The estimate of $  \widetilde{H}_{k,j_1,j_2}^2 (s_1,s_2)$.

Recall (\ref{april28eqn16}) and the definition of kernel $\widetilde{K}^1_{k, n,l  }(y, V(s), v)$ in (\ref{april28eqn22}). As a result of direct computations, the following estimate holds for the kernel after doing integration by parts in $\xi$ in $\theta_{V(s)}(v)$ direction and $(\theta_{V(s )}(v))^{\bot}$ directions, 
\[ 
\big| \widetilde{K}^1_{k, n,l  }(y, V(s), v)\big| \lesssim 2^{k+l-l-n} \big( 2^{-l-n-(1-\epsilon )M_t} |\nabla_x\phi(s, X(s)| + 2^{-\alpha-(1-\epsilon )M_t} |\slashed V(s)| |\slashed X(s)|^{-1} \big)  \psi_{ n }(|\tilde{v}-\tilde{V}(s)|)
\]
\[
\times (1+2^{k+l}|y\cdot\theta_{V(s)}(v) |)^{- N_c^3 } (1+2^{k }|y\times \theta_{V(s)}(v) |)^{-  N_c^3 }\psi_{[  (M_t-17j_1)/13-2\epsilon M_t, 2j_1-(2-15\epsilon) M_t ]}(\slashed X(s ))\]
\[
\lesssim2^{k-n+10\epsilon M_t} \big( 2^{-l-n- M_t} |\nabla_x\phi(s, X(s)| + 2^{-\alpha +n } 2^{(17j_1-M_t)/13  }  \big) (1+2^{k+l}|y\cdot\theta_{V(s)}(v) |)^{- N_c^3 }\]
\[
\times   (1+2^{k }|y\times \theta_{V(s)}(v) |)^{-  N_c^3 }\psi_{[  (M_t-17j_1)/13-2\epsilon M_t, 2j_1- (2+15\epsilon) M_t ]}(\slashed X(s)).
\]
From the above estimate of kernel, the estimate (\ref{2022jan26eqn41}), the cylindrical symmetry of the distribution function,    the obtained estimate (\ref{june2eqn31}), and the estimate (\ref{jan12eqn36}) in Proposition  \ref{spacetimeestimate}, we have 
\[
\big|\widetilde{H}_{k,j_1,j_2}^2 (s_1,s_2)\big|\lesssim \sum_{ n\in[l_1,2]\cap \mathbb{Z}}    \sum_{ l\in[l_2,2]\cap \mathbb{Z}} 2^{k-n+13\epsilon M_t} 2^{-k-\alpha }  2^{(17j_1-M_t)/13  } \big( 2^{(1+6\epsilon)M_t}  2^{-l-n- M_t}2^{-j_2}  + 2^{-\alpha +n }  2^{-2j_1+j_2} \big)
\] 
 
\be\label{june2eqn53}
 \lesssim  2^{-2l_1-l_2-\alpha +30\epsilon M_t} 2^{-j_2+ (17j_1-M_t)/13  } + 2^{-2\alpha }2^{-2j_1+j_2 +  (17j_1-M_t)/13   +30\epsilon M_t} \lesssim 2^{(1-3\epsilon)M_t}. 
\ee

Recall the decomposition (\ref{june2eqn54}), from the estimates (\ref{june2eqn51}), (\ref{june2eqn51}), and (\ref{june2eqn51}), we have
\[
\big| \int_{s_1}^{s_2} H_{k,j_1,j_2}^3(s)d s\big| \lesssim 2^{(1-3\epsilon)M_t}. 
\]
Recall the decomposition (\ref{june1eqn52}). Our desired estimate (\ref{june1eqn30}) holds from the above estimate and the  estimates (\ref{june2eqn61}) and (\ref{june2eqn11}).  
\end{proof}

\section{Proof of main theorems}\label{proofofthe}
As summarized in Proposition \ref{majority} and  Proposition \ref{majorityset},  once we have good control of the majority set,  we can show  the sub-linearity of the high momentum.  The proof of Theorem \ref{maintheoremradial}  and Theorem \ref{maintheorem} will follow in the exactly same sprite. For the sake of readers, we  still provide detailed proof for both of them here. 

\subsection{Proof of Theorem \ref{maintheoremradial}}

Based on the possible size of ``$|v|$'', we decompose $M_n(t)$ into three parts as follows, 
\[
\int_{\R^3}  \int_{\R^3} (1+|v|)^{n_r} |f(t,x,v)| d x dv  
 = \int_{\R^3 }\int_{|v|\geq   ( \tilde{M}_n^r(t)  )^{  (5+4\delta)/  ((6-2\delta)(n-1)  )}  }(1+|v|)^{n_r} |f(t,x,v)| d x dv,
\]
\[
+  \int_{\R^3}\int_{|v|\leq ( \tilde{M}^r_n(t)  )^{  (5+4\delta)/  ((6-2\delta)(n-1)  ) } }(1+|v|)^{n_r} |f(t,x,v)| d x dv.
\]

 Note that, from the relation (\ref{jan13eqn11}) in Lemma \ref{majority}, we have  $| X(0;t,x,v) |+|V(0;t,x,v) |\gtrsim  (\tilde{M}^r_n(t))^{1/(2{n_r})}$ if either    $|x|\gtrsim  (\tilde{M}^r_n(t))^{(1+\delta)/(2{n_r})}$ or $|v|\gtrsim \big( \tilde{M}^r_n(t) \big)^{  (5+3\delta)/  ((6-2\delta)({n_r}-1)  ) }$. If $|x|\gtrsim (\tilde{M}^r_n(t))^{(1+\delta)/(2{n_r})}$, then we have
\[
|X(0;t,x,v)|\geq |x| - t\geq |x|-(\tilde{M}^r_n(t))^{1/(2{n_r})}|\gtrsim (1+|x|).
\]
Therefore, from the above estimate and the assumption on the initial data in (\ref{assumptiononinitialdata}),   the following estimate holds if $|v|\gtrsim \big( \tilde{M}^r_n(t) \big)^{  (5+3\delta)/  ((6-2\delta)( {n_r}-1)  ) }$  regardless the size of $|x|$, 
\be\label{feb1eqn2}
|f(t,x,v)|= |f_0(X(0;t,x,v),V(0;t,x,v))|\leq (\tilde{M}^r_n(t))^{-4  }(1+|x|)^{-4}. 
\ee
Moreover,  if  $|v| \gtrsim   \big( \tilde{M}^r_n(t) \big)^{5/(2n_r)}$, from the equation (\ref{jan12eqn21}) and the estimate (\ref{jan31eqn11}),  we have 
\[
|V(0;t,x,v)|\gtrsim |v|- \int_{0}^t \big(1+ \big( \tilde{M}^r_n(t) \big)^{  (5+\delta)/  ( (3-\delta)({n_r}-1)  )}\big) d s \gtrsim (1+ |v|). 
\]
From the above two estimates and the assumption on the initial data in (\ref{assumptiononinitialdata}),   the following estimate holds if $|v|\gtrsim   \big( \tilde{M}^r_n(t) \big)^{  (5+3\delta)/  ((6-2\delta)({n_r}-1)  ) } $  regardless the size of $|x|$, 
\[
|f(t,x,v)|=  |f_0(X(0;t,x,v),V(0;t,x,v))|\lesssim   (1+|x|  )^{-4} (1+|v |)^{- {n_r}-4}.
\]
Therefore, from  the above estimate, we have
\[
  \int_{\R^3 }\int_{|v|\geq   ( \tilde{M}^r_n(t)  )^{  (5+4\delta)/  ((6-2\delta)(n-1)  )}  }(1+|v|)^{{n_r}} |f(t,x,v)| d x dv \lesssim 1.
\] 
 
From the conservation law (\ref{conservationlaw}), we have 
\[
\int_{\R^3}\int_{|v|\leq ( \tilde{M}^r_n(t)  )^{  (5+4\delta)/  ((6-2\delta)({n_r}-1)  ) } }(1+|v|)^{n_r} |f(t,x,v)| d x dv\lesssim \big( \tilde{M}^r_n(t) \big)^{  (5+4\delta){n_r}/  ((6-2\delta)({n_r}-1)  ) }.
\]
To sum up, we have
\[
M_{n_r}(t)\lesssim \big(\tilde{M}^r_n(t) \big)^{  (5+4\delta){n_r}/  ((6-2\delta)({n_r}-1)  ) }. 
\]
Since the above estimate holds for any $t \in[0,T^{\ast})$ and $\tilde{M}^r_n(t)$ is an increasing function with respect to $t$, the following estimate holds for any $s\in [0, t]$, 
\[
M_{n_r}(s)\lesssim \big( \tilde{M}_{n }^r (s) \big)^{  (5+4\delta)n_r/  ((6-2\delta)(n_r-1)  ) }\leq \big( \tilde{M}_n^r(  t) \big)^{  (5+4\delta)n_r/  ((6-2\delta)(n_r-1)  ) }.
\]
Hence
\[
\tilde{M}^r_n(t)=\sup_{s\in [0, t]}M_{n_r}(s)+  (1+t)^{2n_r}  \lesssim \big( \tilde{M}_{n_r}(t) \big)^{  (5+4\delta)n_r/  ((6-2\delta)(n_r-1)  ) } +  (1+t)^{2 n_r}, 
 \quad \Longrightarrow \tilde{M}^r_n(t)\lesssim (1+t)^{  2n_r}. 
\]
Therefore, from the estimate (\ref{jan31eqn11}) in Lemma \ref{roughcontrol}, we have
\[
\|\nabla_x\phi(t,x)\|_{L^\infty_x} \lesssim  (1+t)^{  2 n_r }.
\]
We have shown the desired fact that $\nabla_x\phi\in L^\infty([0,T^{\ast})\times \R_{x}^3)$.  Hence finishing the proof of Theorem \ref{maintheoremradial}. 
 
\qed

\subsection{Proof of Theorem \ref{maintheorem}}
From the conservation law (\ref{conservationlaw}), we have 
\be\label{march20eqn11}
 \big| \int_{\R^3}\int_{|v|\leq 2^{(1- \epsilon/3)M_t }}(1+|v|)^{n_c} f(t,x,v)   d x dv\big|\lesssim  2^{(n_c-1)(1- \epsilon/3)M_t}\leq (\tilde{M}_{n}^c (t))^{ 1- \epsilon/3 }. 
\ee
Recall the definition of the majority set $R^{cyl}(t,t)$ in (\ref{may10majority}) and the definition of $\beta_t$ in (\ref{vlasovpo}). From the   estimate   (\ref{may11eqn61}) in Proposition \ref{majorityset}, we know that  $|X(0;t,x ,v ) |+|  V(0;t,x ,v ) |\geq 2^{  M_t /2}$ if $|v|\geq 2^{(1- \epsilon/3)M_t }$. From the decay assumption of the initial data in (\ref{assumcylinrical}) and the following estimate holds if $|v|\geq 2^{(1- \epsilon/3)M_t }$, 
\be\label{may11eqn31}
|f(t,x,v)|=| f_0(X(0;t,x ,v ), V(0;t,x ,v ))| \lesssim (1+|X(0;t,x ,v )|+|V(0;t,x ,v )|)^{-N_c+10}\lesssim 2^{-4 n_c M_t }.
\ee
Moreover, recall (\ref{characteristicseqn}), from the   $ L^\infty_x$-type estimate of electric field in (\ref{june2eqn71}), we have
\be\label{may11eqn32}
\big||v|-|V(0;t,x,v)|\big| \lesssim    2^{5M_t/3+2\epsilon M_t}, \quad \Longrightarrow |v|\sim |V(0;t,x,v)|,\quad  \textup{if\,\,} |v|\gtrsim 2^{2 M_t},
\ee  
\be\label{may11eqn33}
\big||x|-|X(0;t,x,v)|\big| \lesssim    2^{\epsilon M_t}, \quad \Longrightarrow |x|\sim |X(0;t,x,v)|,\quad  \textup{if\,\,} |x|\gtrsim 2^{2\epsilon M_t}.
\ee 
To sum up, after combining the above estimates (\ref{may11eqn31}--\ref{may11eqn33}), we know that the following estimate holds if $|v|\geq 2^{(1- \epsilon/3)M_t }$, 
\[
|f(t,x,v)|=|f_0(X(0;t,x,v),V(0;t,x,v))|\lesssim (1+|x|)^{-4}(1+|v |)^{-n_c -4}.  
\]
From the above estimate, we have
\be\label{may11eqn34}
 \big| \int_{\R^3}\int_{|v|\geq 2^{(1- \epsilon/3)M_t }}(1+|v|)^{n_c} f(t,x,v)   d x dv\big|\lesssim 1.
\ee

Therefore, recall (\ref{may2eqn1}), from the estimates (\ref{march20eqn11}) and (\ref{may11eqn34}), we know that the following estimate holds for any  $t\in[0,T)$,
\be 
M_{n_c}(t)\lesssim \big( \tilde{M}^c_{n}(t) \big)^{  1- \epsilon/3 }.  
\ee
From the above estimate and 
the fact that $\tilde{M}_n(t)$ is an increasing function with respect to $t $, the following estimate holds for any $s\in [0, t]$,
\[
\tilde{M}^c_{n}(t)=\sup_{s\in [0, t]}M_{n_c}(s)+  (1+t)^{n_c^2}  \lesssim \big( \tilde{M}_{n}^c(  t) \big)^{  1- \epsilon /3}+  (1+t)^{n_c^2}, 
 \quad \Longrightarrow \tilde{M}_{n}^c(t)\lesssim   (1+t)^{n_c^2}. 
\]
From the   $L^\infty_x$ estimate of electric field in (\ref{june2eqn71}), we know that   the desired fact that $\nabla_x\phi\in L^\infty([0,T^{\ast})\times \R_{x}^3)$ holds.  Hence finishing the proof of Theorem \ref{maintheorem}.

\end{document}